\documentclass[ssy]{imsart}
\usepackage{amsmath,amsthm,verbatim,amssymb,amsfonts,amscd, graphicx}
\usepackage{amssymb,bm}
\usepackage{mathrsfs}
\usepackage{hyperref}
\hypersetup{
    colorlinks,
    citecolor=black,
    filecolor=black,
    linkcolor=blue,
    urlcolor=black
}
\usepackage{url}
\usepackage{verbatim}
\usepackage{constants}
\usepackage{enumerate}
\newcommand{\beas}{\begin{eqnarray*}}
\newcommand{\enas}{\end{eqnarray*}}
\newcommand{\bea}{\begin{eqnarray}}
\newcommand{\ena}{\end{eqnarray}}
\newcommand{\bms}{\begin{multline*}}
\newcommand{\ems}{\end{multline*}}
\newcommand{\bels}{\begin{align*}}
\newcommand{\enls}{\end{align*}}
\newcommand{\bel}{\begin{align}}
\newcommand{\enl}{\end{align}}

\newcommand{\ignore}[1]{}

\newtheorem{theorem}{Theorem}[section]

\newtheorem{remark}{Remark}[section]
\newtheorem{lemma}{Lemma}[section]
\newtheorem{definition}{Definition}[section]

\newtheorem{assumption}{Assumption}[section]

\def\blfootnote{\xdef\@thefnmark{}\@footnotetext}

\newcommand{\expect}[1]{\mathbb{E}{\l(#1\r)}}

\newcommand{\dotp}[2]{\left\langle#1,#2\right\rangle}

\def\r{\right}
\def\l{\left}

\begin{document}

\begin{frontmatter}
\title{
Asynchronous Optimization over Weakly Coupled Renewal Systems
}
\runtitle{Asynchronous renewal optimization}
\begin{aug}
\author{\fnms{Xiaohan} \snm{Wei}\thanksref{t1}\ead[label=e1]{xiaohanw@usc.edu}}
\and
\author{\fnms{Michael} \snm{J. Neely}\thanksref{t1}\ead[label=e2]{mikejneely@gmail.com}}
\thankstext{t1}{Department of Electrical Engineering, University of Southern California}
\thankstext{t2}{This work was supported by grant NSF CCF-1718477.}
\runauthor{X. Wei, M. J. Neely}

\affiliation{University of Southern California}
\printead{e1,e2}
\end{aug}

\maketitle

\begin{abstract}
This paper considers optimization over multiple renewal systems coupled by time average constraints.  These systems act asynchronously over variable length frames. When a particular system starts a new renewal frame, it chooses an action from a set of options for that frame. The action determines the duration of the frame, the penalty incurred during the frame (such as energy expenditure), and a vector of performance metrics (such as instantaneous number of job services).  The goal is to minimize the time average penalty subject to time average overall constraints on the corresponding metrics. This problem has applications to task processing networks and coupled Markov decision processes (MDPs).  We propose a new algorithm so that each system can make its own decision after observing a global multiplier that is updated every slot.  We show that this algorithm satisfies the desired constraints and achieves $O(\epsilon)$ near optimality with $O(1/\epsilon^2)$ convergence time.
\end{abstract}


\end{frontmatter}

\section{Introduction}
Consider $N$ renewal systems that operate over a slotted timeline ($t \in \{0, 1, 2, \ldots\}$).  
The timeline for each system $n \in \{1, \ldots, N\}$ is segmented into back-to-back intervals of time slots called \emph{renewal frames}. The start of each renewal frame for a particular system is called a \emph{renewal time} or simply a \emph{renewal} for that system.
The duration of each renewal frame is a random positive integer with distribution that depends on a control action chosen by the system at the start of the frame.  The decision at each renewal frame also determines the penalty and a vector of performance metrics during this frame.  The systems are coupled by time average constraints placed on these metrics over all systems.  The goal is to design a decision strategy for each system so that overall time average penalty is minimized subject to time average constraints. 

We use $k=0, 1,2,\cdots$ to index the renewals. Let $t^n_k$ be the time slot corresponding to the $k$-th renewal of the $n$-th system with the convention that $t^n_0=0$. Let $\mathcal{T}^{n}_{k}$ be the set of all slots from $t^n_k$ to $t^n_{k+1}-1$.
At time $t^n_k$, the $n$-th system chooses a possibly random decision $\alpha^n_k$ in a set $\mathcal{A}^n$. This action determines the distributions of the following random variables:
\begin{itemize}
\item The duration of the $k$-th renewal frame $T^n_k:=t^n_{k+1}-t^n_k$, which is a positive integer.
\item A vector of performance metrics at each slot of that frame 
$\mathbf{z}^n[t]:=\left(z^n_1[t],~z^n_2[t],~\cdots,~z^n_L[t]\right)$,\\
$t\in\mathcal{T}^{n}_{k}$.
\item A penalty incurred at each slot of the frame $y^n[t]$, $t\in\mathcal{T}^{n}_{k}$.
\end{itemize}
We assume each system has the \textit{renewal property} that given $\alpha^n_k=\alpha^n\in\mathcal{A}^n$, the random variables $T^n_k$, $\mathbf{z}^n[t]$ and $y^n[t]$,~$t\in\mathcal{T}^n_k$ are independent of the information of all systems from the slots before $t^n_k$ with the following \textit{known} conditional expectations $\expect{\left.T^n_k\right|\alpha^n_k=\alpha^n}$, $\expect{\left.\sum_{t\in\mathcal{T}^n_k}y^n[t]\right|\alpha^n_k=\alpha^n}$ and $\expect{\left.\sum_{t\in\mathcal{T}^n_k}\mathbf{z}^n[t]\right|\alpha^n_k=\alpha^n}$. 


In addition, we have an uncontrollable external i.i.d. random process $\{\mathbf{d}[t]\}_{t=0}^{\infty}\subseteq\mathbb{R}^L$ which can be observed during each time slot. Let $d_l:=\expect{d_l[t]}$. 
 The goal is to minimize the total time average penalty of these $N$ renewal systems subject to $L$ total time average constraints on the performance metrics related to the external i.i.d. process, i.e. we aim to solve the following optimization problem:
\begin{align}
\min~~&\limsup_{T\rightarrow\infty}\frac{1}{T}\sum_{t=0}^{T-1}\sum_{n=1}^N\expect{y^n[t]}\label{prob-1}\\
\textrm{s.t.}~~ &\limsup_{T\rightarrow\infty}\frac{1}{T}\sum_{t=0}^{T-1}\sum_{n=1}^N\expect{z^n_l[t]}\leq d_l,~~l\in\{1,2,\cdots,L\}.\label{prob-2}
\end{align} 
This problem is challenging because these $N$ systems are weakly coupled by the time average constraints \eqref{prob-2}, yet each of them operates over its own renewal frames. The renewals of different systems do not have to be synchronized and they do not have to occur at the same rate. Fig. \ref{fig:Stupendous1} plots a sample timeline of three parallel renewal systems.
In Section \ref{section:algorithm}, we will develop an algorithm that does not need the knowledge of $d_l=\expect{d_l[t]}$ with a provable performance guarantee.

\begin{figure}[htbp]
   \minipage{0.6\textwidth}
   \includegraphics[width=\linewidth]{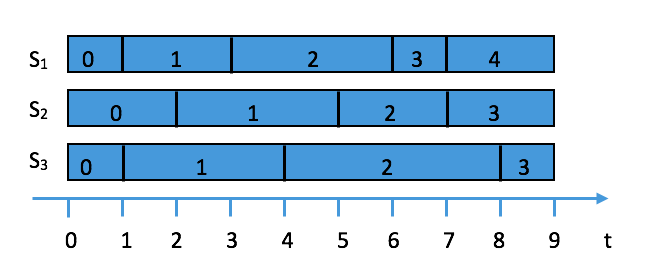} 
   \caption{The sample timelines of three asynchronous parallel renewal systems, where the numbers underneath the figure index time slots and the numbers inside the blocks index the renewals of each system.}
   \label{fig:Stupendous1}
   \endminipage
\end{figure}

\subsection{Example Applications}\label{sec:application}

\subsubsection{Multi-server energy-aware scheduling}\label{sec:server-app}
Consider a slotted time system with $L$ classes of jobs and $N$ servers. Job arrivals are Poisson distributed with rates $\lambda_1,~\cdots,~\lambda_L$, respectively. These jobs are stored in separate queues denoted as $Q_1[t],~\cdots,~Q_L[t]$ in a router waiting to be served. Assume the system is empty at time $t=0$ so that $Q_l[0]=0,~\forall l\in\{1,2,\cdots,L\}$. Let $\lambda_l[t]$ be the precise number of class $l$ job arrivals at slot $t$, then, we have $\expect{\lambda_l[t]}=\lambda_l,~\forall l\in\{1,2,\cdots,L\}$. 
Let $\mu^n_l[t]$ and $e^n[t]$ be the number of class $l$ jobs served and the energy consumption for server $n$ at time slot $t$, respectively. Fig. \ref{fig:multi-server} sketches an example architecture of the system with 3 classes of jobs and 10 servers.

Each server makes decisions over renewal frames and the first frame starts at time slot $t=0$. Successive renewals can happen at different slots for different servers. For the $n$-th server, at the beginning of the $k$-th frame ($k\in\mathbb{N}$), it chooses a processing mode $m^n_k$ within the set of all modes $\mathcal{M}^n$. 
The processing mode $m_k^n$ determines distributions on the number of jobs served, the service time, and the energy expenditure, with conditional expectations: 
\begin{itemize}
\item $\widehat{T}^n(m^n_k):=\expect{\left.T^n_k\right|~m^n_k}$. The expected frame size. 
\item
$\widehat{\mu}^n_l(m^n_k)= \expect{\left.\sum_{t\in\mathcal{T}^n_k}\mu^n_l[t]\right|~m^n_k}$. The expected number of class $l$ jobs served. 
\item $\widehat{e}^n(m^n_k)= \expect{\left.\sum_{t\in\mathcal{T}^n_k}e^n[t]\right|~m^n_k}$. The expected energy consumption.
\end{itemize}

The goal is to minimize the time average energy consumption, subject to the queue stability constraints, i.e.
\begin{align}
\min~~&\limsup_{T\rightarrow\infty}\frac1T\sum_{t=0}^{T-1}\sum_{n=1}^N\expect{e^n[t]}\label{ews-1}\\
s.t.~~&\liminf_{T\rightarrow\infty}\frac1T\sum_{t=0}^{T-1}\sum_{n=1}^N\expect{\mu^n_l[t]}
\geq\lambda_l,~\forall l\in\{1,2,\cdots,L\}.\label{ews-2}
\end{align}
Thus, we have formulated the problem into the form \eqref{prob-1}-\eqref{prob-2}.
Note that the external process in this example is the arrival process of $L$ classes of jobs with potentially unknown arrival rates $\lambda_l$.

Previously, \cite{Neely12} treats a special case of this problem where all energy and service quantities are deterministic functions of the processing modes.  The newly developed algorithm in the current paper can be used to solve this problem with considerably more general stochastic assumptions.

\begin{figure}[htbp]
   \minipage{0.6\textwidth}
   \includegraphics[width=\linewidth]{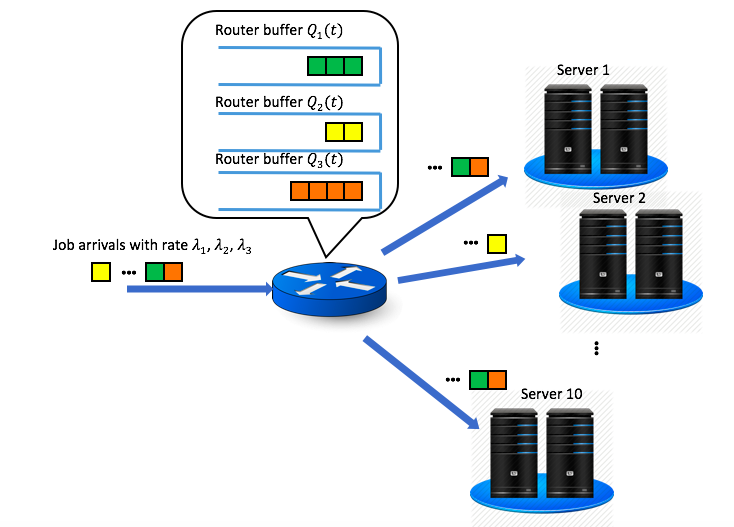} 
   \caption{Illustration of an energy-aware scheduling system with 3 classes of jobs and 10 parallel servers.}
   \label{fig:multi-server}
   \endminipage
\end{figure}

\subsubsection{Coupled ergodic MDPs}\label{sec:MDP}
Consider $N$ discrete time Markov decision processes (MDPs) over an infinite horizon.  Each MDP
consists of a finite state space $\mathcal{S}^n$, and an action space $\mathcal{U}^n$ at each state $s\in\mathcal{S}^n.$\footnote{To simplify the notation, we assume each state has the same action space $\mathcal{A}^n$. All our analysis generalizes trivially to states with different action spaces.} For each state $s\in\mathcal{S}$, we use $P_u^n(s,s')$ to denote the transition probability from $s\in\mathcal{S}^n$ to $s'\in\mathcal{S}^n$ when taking action $u\in\mathcal{U}^n$, i.e. 
\[
P_u^n(s,s')  = Pr(s[t+1]=s'~|~s[t] = s,~u[t]=u),
\]
where $s[t]$ and $u[t]$ are state and action at time slot $t$. 

At time slot $t$, after observing the state $s[t]\in\mathcal{S}^n$ and choosing the action $u[t]\in\mathcal{U}^n$, the n-th MDP receives a penalty $y^n(u[t],s[t])$ and $L$ types of resource costs $z_{1}^n(u[t],s[t]),\cdots,z_{L}^n(u[t],s[t])$, where these functions are all bounded mappings from $\mathcal{S}^n\times\mathcal{U}^n$ to $\mathbb{R}$. For simplicity we write $y^n[t] = y^n(u[t],s[t])$ and $z_l^n[t] = z_l^n(u[t],s[t])$. The goal is to minimize the time average overall penalty with constraints on time average overall costs, where these MDPs are weakly coupled through the time average constraints. This problem 
can be written in the form \eqref{prob-1}-\eqref{prob-2}.

In order to define the renewal frame, we need one more assumption on the MDPs. We assume each of the MDPs is \textit{ergodic}, i.e. there exists a state which is recurrent and the corresponding Markov chain is aperiodic under any randomized stationary policy\footnote{A \textit{randomized stationary policy} $\pi$ is an algorithm which chooses actions at state $s\in\mathcal{S}^n$ according to a fixed conditional distribution $\pi(u | s),~u\in\mathcal{U}^n$ and is independent of all other past information, i.e. $Pr(u[t] |\mathcal{F}_t) = \pi(u[t] |s[t])$,~$u[t]\in\mathcal{U}^n$, $s[t]\in\mathcal{S}^n$ and $\mathcal{F}_t$ is the past information up to time 
$t$.}, 
with bounded expected recurrence time. Under this assumption, the renewals for each MDP can be defined as successive revisitations to the recurrent state, and the action set $\mathcal{A}^n$ in such scenario is defined as the set of all randomized stationary policies that can be implemented in one renewal frame. 
Thus, our formulation includes coupled ergodic MDPs. We refer to \cite{Al99}, \cite{Be01}, and \cite{Ro02} for more details on MDP theory and related topics.

As a side remark, this multi-MDP problem can be viewed as a single MDP on an enlarged state space. Constrained MDPs are discussed previously in \cite{Al99}.
One can show that under the previous ergodic assumption, the minimum of \eqref{prob-1}-\eqref{prob-2} is achieved by a randomized stationary policy, and furthermore, such a policy can be obtained via solving a linear program reformulated from \eqref{prob-1}-\eqref{prob-2} offline. However, formulating such LP requires the knowledge of all the parameters in the problem, including the statistics of the external process $\{\mathbf{d}[t]\}_{t=0}^\infty$, and the resulting LP is often computationally intractable when the number of MDPs is very large. On the contrary,
our newly developed algorithm is carried out in an online manner, does't require the statistics of the external process and
enjoys a natural ``decoupled'' structure, effectively reducing the computational load.


\subsection{Challenges and previous approaches}
As mentioned above, for the special case of coupled ergodic MDPs, this problem can be solved via a linear program (see \cite{Al99} and also \cite{Fo66} for detailed discussions on formulating MDPs as linear programs). 
However, 
this approach becomes intractable as the number of MDPs gets very large. 
On the other hand, existing asynchronous algorithms and analysis (e.g. \cite{BT97}\cite{BGPS06}\cite{PXYY15}\cite{SN11}) treat only the case where system delays (frames) are of fixed distribution independent of the actions or even deterministic, which are not readily extendable to our problem.

The main technical challenge is the dilemma on how to pick a correct time scale to carry out an algorithm and corresponding analysis. On one hand, since time is slotted, one would naturally think of synchronizing all systems on the slot scale and designing a slot-based algorithm. However, since each renewal spans multiple slots, any such algorithm would essentially cut some renewals in the middle and it would be difficult to analyze any particular system. On the other hand, if we analyze each system over its own renewal frames, it is not clear how to piece together these individual analyses.

Prior approaches treat this challenge only in special cases. 
The works \cite{Neely12} and \cite{Ne12} consider a special case where all quantities introduced above are deterministic functions of the actions. 
The work \cite{WN15} considers another special case of the current formulation in server scheduling, where there is only one queue stability constraint and it can be easily met via controlling the arrival rate to the system. These two methods circumvent the aforementioned dilemma by making extra assumptions on the system and thus can not be generalized to the current setting.

\subsection{Our contributions}
The current paper develops a new algorithm where each system operates on its own renewal frame. It is fully analyzed with convergence as well as convergence time results. 
 As a first technical contribution, we fully characterize the fundamental performance region of the problem \eqref{prob-1}-\eqref{prob-2} (Lemma \ref{stationary-lemma}).
 We then resolve the aforementioned dilemma by constructing a supermartingale along with a stopping-time to ``synchronize'' all systems on a slot basis, by which we could piece together analysis of each individual system to prove the convergence of the proposed algorithm. Furthermore, encapsulating this new idea into convex analysis tools, we
 prove the $\mathcal{O}(1/\varepsilon^2)$ convergence time of the proposed algorithm to reach $\mathcal{O}(\varepsilon)$ near optimality under a mild assumption on the existence of a Lagrange multiplier (Section \ref{sec-convergence-time}). Specifically, we show that for any accuracy $\epsilon>0$ and any time
$T\geq1/\varepsilon^2$, the sequence $\{y^n[t]\}$ and $\{\mathbf{z}^n[t]\}$ produced by our algorithm satisfies,
\begin{align*}
&\frac1T\sum_{t=0}^{T-1}\sum_{n=1}^N\expect{y^n[t]}\leq f_*+\mathcal{O}(\varepsilon),\\
&\frac1T\sum_{t=0}^{T-1}\sum_{n=1}^N\expect{z^n_l[t]}\leq d_l+\mathcal{O}(\varepsilon),
l\in\{1,2,\cdots,L\},
\end{align*}
where $f_*$ denotes the optimal objective value of \eqref{prob-1}-\eqref{prob-2}.
Simulation experiments on the aforementioned multi-server energy-aware scheduling problem also demonstrate the effectiveness of the proposed algorithm.

\subsection{Other related works}
The problem considered in the current paper is a generalization of optimization over a single renewal system. It is shown in \cite{Ne09} that for the single renewal system with finite action set, the problem can be solved (offline) via a linear fractional program. Methods for solving linear fractional programs can be found in \cite{BV04} and \cite{Sc83}. The \textit{drift-plus-penalty ratio} approach is also developed in  \cite{Neely2010} and \cite{Ne09} for the single renewal system. 

On the other hand, our problem is also related to the multi-server scheduling as is shown in one of the example applications. When assuming proper statistics of the arrivals and/or services, energy optimization problems in multi-server systems can also be treated via queueing theory. Specifically, by assuming both arrivals and services are Poisson distributed, \cite{GDHS13} treats the multi-server system as an M/M/k/setup queue and explicitly computes several performance metrics via the renewal reward theorem. By assuming arrivals are Poisson and only one server, \cite{LN14} and \cite{Yao02} treat the system as a multi-class M/G/1 queue and optimize the average energy consumption via polymatroid optimization.

\subsection{Notation and organization of the paper}
Throughout the paper, we use superscript $n\in\{1,2,\cdots,N\}$ to index different systems, use the subscript $l\in\{1,2,\cdots,L\}$ to index different constraints and use the subscript $k\in\mathbb{N}$ to index the frames. For any vector $\mathbf{x}\in\mathbb{R}^d$, the considered norms are $\|\mathbf{x}\|:=\sqrt{\sum_{i=1}^dx_i^2}$, $\|\mathbf{x}\|_1:=\sum_{i=1}^d|x_i|$ and $\|\mathbf{x}\|_{\infty}:=\max_i~|x_i|$.

The rest of the paper is organized as follows: Section \ref{section:algorithm} introduces the proposed algorithm along with technical assumptions. Section \ref{section:limiting} introduces our main technical argument proving the convergence of the proposed algorithm via supermartingale and stopping time constructions. Building upon these technical tools, Section \ref{sec-convergence-time} takes one step further and proves the convergence time of the proposed algorithm. Finally, a simulation study regarding multi-server energy-aware scheduling is given in Section
\ref{section-application}.


%


\section{Algorithm}\label{section:algorithm}

\subsection{Technical preliminaries}
Throughout the paper, we make the following assumptions.

\begin{assumption}\label{feasible-assumption}
The problem \eqref{prob-1}-\eqref{prob-2} is feasible, i.e. there are action sequences 
$\{\alpha_k^n\}_{k=0}^\infty$ for all $n\in\{1,2,\cdots,N\}$ so that the corresponding process $\{\mathbf{z}^n[t]\}_{t=0}^\infty$ satisfies the constraints \eqref{prob-2}.
\end{assumption}

Following this assumption, we define $f_*$ as the infimum objective value for \eqref{prob-1}-\eqref{prob-2} over all decision sequences that satisfy the constraints.

\begin{assumption}[Boundedness]\label{bounded-assumption}
For any $k\in\mathbb{N}$ and any $n\in\{1,2,\cdots,N\}$, there exist absolute constants $y_{\max}$, $z_{\max}$ and $d_{\max}$ such that
\begin{align*}
|y^n[t]|\leq y_{\max},~~|z^n_l[t]|\leq z_{\max},~~|d_l[t]|\leq d_{\max},~~\forall t\in\mathcal{T}^n_k,
~\forall l\in\{1,2,\cdots,L\}.
\end{align*}
Furthermore, 
there exists an absolute constant $B\geq1$ such that for every fixed $\alpha^n\in\mathcal{A}^n$ and every $s\in\mathbb{N}$ for which $Pr(T^n_k\geq s|\alpha^n_k=\alpha^n)>0$,
\begin{equation}\label{residual-life-bound}
\expect{\left.(T^n_k-s)^2\right|~\alpha^n_k=\alpha^n,T^n_k\geq s}\leq B.
\end{equation}
\end{assumption}
\begin{remark}
The quantity $T^n_k-s$ is usually referred to as the residual lifetime. In the special case where $s=0$, \eqref{residual-life-bound} gives the uniform second moment bound of the renewal frames as
\[\expect{\left.(T^n_k)^2\right|~\alpha^n_k=\alpha^n}\leq B.\]
Note that \eqref{residual-life-bound} is satisfied for a large class of problems. In particular, it can be shown to hold in the following three cases:
\begin{enumerate}
\item If the inter-renewal $T^n_k$ is deterministically bounded.
\item If the inter-renewal $T^n_k$ is geometrically distributed.
\item If each system is a finite state ergodic MDP with a finite action set.
\end{enumerate}
\end{remark}

\begin{definition}\label{PV-def}
For any $\alpha^n\in\mathcal{A}^n$, let
\[\widehat{y}^n(\alpha^n):=\expect{\left.\sum_{t\in\mathcal{T}^n_k}y^n[t]\right|\alpha^n_k=\alpha^n},~
~\widehat{z}^n_l(\alpha^n):=\expect{\left.\sum_{t\in\mathcal{T}^n_k}z^n_l[t]\right|\alpha^n_k=\alpha^n},\]
and $\widehat{T}^n(\alpha^n):=\expect{T^n_k|\alpha^n_k=\alpha^n}$. Define
\begin{align*}
&\widehat{f}^n(\alpha^n):=\widehat{y}^n(\alpha^n)/\widehat{T}^n(\alpha^n),\\
&\widehat{g}^n_l(\alpha^n):=\widehat{z}^n_l(\alpha^n)/\widehat{T}^n(\alpha^n),~\forall l\in\{1,2,\cdots,L\},
\end{align*}
and let $\left(\widehat{f}^n(\alpha^n),~\widehat{\mathbf{g}}^n(\alpha^n)\right)$ be a performance vector under the action $\alpha^n$.
\end{definition}

Note that by Assumption \ref{bounded-assumption}, $\widehat{y}^n(\alpha^n)$ and $\widehat{\mathbf{z}}^n(\alpha^n)$ in Definition \ref{PV-def} are both bounded, and $T^n_k\geq1,~\forall k\in\mathbb{N}$, thus, the set $\left\{\left(\widehat{f}^n(\alpha^n),~\widehat{\mathbf{g}}^n(\alpha^n)\right),~\alpha^n\in\mathcal{A}^n\right\}$ is also bounded. The following mild assumption 
states that this set is also closed.

\begin{assumption}\label{compact-assumption}
 The set $\left\{\left(\widehat{f}^n(\alpha^n),~\widehat{\mathbf{g}}^n(\alpha^n)\right),~\alpha^n\in\mathcal{A}^n\right\}$ is compact.
\end{assumption}

Finally, we define the performance region of each individual system as follows.

\begin{definition}\label{PR-def}
Let $\mathcal{S}^n$ be the convex hull of $\left\{\left(\widehat{y}^n(\alpha^n),~\widehat{\mathbf{z}}^n(\alpha^n),~\widehat{T}^n(\alpha^n)\right):~\alpha^n\in\mathcal{A}^n\right\}\subseteq\mathbb{R}^{L+2}$. Define
\[\mathcal{P}^n:=\left\{\left(y/T,~\mathbf{z}/T\right):~(y,\mathbf{z},T)\in\mathcal{S}^n\right\}\subseteq\mathbb{R}^{L+1}\]
as the performance region of system $n$.
\end{definition}

\subsection{Proposed algorithm}
In this section, we propose an algorithm where each system can make its own decision after observing a global vector of multipliers which is updated using the global information from all systems. We start by defining a vector of virtual queues $\mathbf{Q}[t]:=\left(Q_1[t],~Q_2[t],~\cdots,~Q_L[t]\right)$, which are 0 at $t=0$ and updated as follows,
\begin{align}
Q_l[t+1]=\max\left\{Q_l[t]+\sum_{n=1}^Nz_l^n[t]-d_l[t],~0\right\},~
l\in\{1,2,\cdots,L\}.\label{queue-update}
\end{align}
These virtual queues will serve as global multipliers to control the growth of corresponding resource consumptions.

Then, the proposed algorithm runs as follows via a fixed trade-off parameter $V>0$:
\begin{itemize}
\item At the beginning of $k$-th frame of system $n$, the system observes the vector of virtual queues $\mathbf{Q}[t^n_k]$ and makes a decision $\alpha^n_k\in\mathcal{A}^n$ so as to solve the following subproblem:
\begin{align}\label{DPP-ratio}
D^n_k:=\min_{\alpha^n\in\mathcal{A}^n}
\frac{\expect{\left.\sum_{t\in\mathcal{T}^n_k}\left(Vy^n[t]+\dotp{\mathbf{Q}[t^n_k]}{\mathbf{z}^n[t]}\right)\right|\alpha^n_k=\alpha^n,\mathbf{Q}[t^n_k]}}{\expect{\left.T^n_k\right|\alpha^n_k=\alpha^n,\mathbf{Q}[t^n_k]}}.
\end{align}
\item Update the virtual queue after each slot:
\begin{align*}
Q_l[t+1]=\max\left\{Q_l[t]+\sum_{n=1}^Nz_l^n[t]-d_l[t],~0\right\},~
l\in\{1,2,\cdots,L\}.
\end{align*}
\end{itemize}
Note that using the notation specified in Definition \ref{PV-def}, we can rewrite \eqref{DPP-ratio} in a more concise way as follows:
\begin{align}\label{DPP-ratio-simple}
\min_{\alpha^n\in\mathcal{A}^n} \left\{V\widehat{f}^n(\alpha^n)+\dotp{\mathbf{Q}[t^n_k]}{\widehat{\mathbf{g}}^n(\alpha^n)}\right\},
\end{align}
which is a deterministic optimization problem. Then,
by the compactness assumption (Assumption \ref{compact-assumption}), there always exists a solution to this subproblem. 

This algorithm requires knowledge of the conditional expectations associated with the performance vectors $\left(\widehat{f}^n(\alpha^n),~\widehat{\mathbf{g}}^n(\alpha^n)\right),~\alpha^n\in\mathcal{A}^n$, but only requires individual systems $n$ to know their own $\left(\widehat{f}^n(\alpha^n),~\widehat{\mathbf{g}}^n(\alpha^n)\right),~\alpha^n\in\mathcal{A}^n$, and therefore decouples these systems. Furthermore, the virtual queue update uses observed $d_l[t]$ and does not require knowledge of distribution or mean of $d_l[t]$.

In addition, we introduce $\mathbf{Q}[t]$ as ``virtual queues'' for the following two reasons: First, it can be mapped to real queues in applications (such as the server scheduling problem mentioned in Section \ref{sec:server-app}), where $\mathbf{d}[t]$ stands for the arrival process and $\mathbf{z}[t]$ is the service process.
Second, stabilizing these virtual queues implies the constraints \eqref{prob-2} are satisfied, as is illustrated in the following lemma, whose proof is given in the appendix.

\begin{lemma}\label{lemma:queue-bound}
If $Q_l[0]=0$ and $\lim_{T\rightarrow\infty}\frac{1}{T}{\expect{Q_l[T]}}=0$, then, $\limsup_{T\rightarrow\infty}\frac{1}{T}\sum_{t=0}^{T-1}\sum_{n=1}^N\expect{z^n_l[t]}\leq d_l$.
\end{lemma}

\subsection{Computing subproblems}

Since a key step in the algorithm is to solve the optimization problem \eqref{DPP-ratio-simple}, we make several comments on the computation of the ratio minimization \eqref{DPP-ratio-simple}.
In general, one can solve the ratio optimization problem \eqref{DPP-ratio} (therefore \eqref{DPP-ratio-simple}) via a bisection search algorithm. For more details, see section 7 of \cite{Neely2010}. However, more often than not, bisection search is not the most efficient one. We will discuss two special cases arising from applications where we can find a simpler way of solving the subproblem.

First of all, when there are only a finite number of actions in the set $\mathcal{A}^n$, one can solve \eqref{DPP-ratio-simple} simply via enumerating. This is a typical scenario in energy-aware scheduling where a finite action set consists of different processing modes that can be chosen by servers.

 Second, when the set $\left\{\left(\widehat{y}^n(\alpha^n),~\widehat{\mathbf{z}}^n(\alpha^n),~\widehat{T}^n(\alpha^n)\right):~\alpha^n\in\mathcal{A}^n\right\}$ specified in Definition \ref{PR-def} is itself a convex hull of a finite sequence 
 $\{(y_j,\mathbf{z}_j,T_j)\}_{j=1}^m$, then, \eqref{DPP-ratio-simple} can be rewritten as a simple enumeration:
 \[
\min_{i\in\{1,2,\cdots,m\}}~~\left\{V\frac{y_i}{T_i}+\dotp{\mathbf{Q}[t^n_k]}{\frac{\mathbf{z}_i}{T_i}}\right\}.
 \]
 To see this, note that by definition of convex hull, for any $\alpha^n\in\mathcal{A}^n$,
 $\left(\widehat{y}^n(\alpha^n),~\widehat{\mathbf{z}}^n(\alpha^n),~\widehat{T}^n(\alpha^n) \right)
 = \sum_{j=1}^mp_j\cdot(y_j,z_j,T_j)$ for some $\{p_j\}_{j=1}^m$, $p_j\geq0$ and $\sum_{j=1}^mp_j=1$. Thus, 
  \begin{align*}
V\widehat{f}^n(\alpha^n)+\dotp{\mathbf{Q}[t^n_k]}{\widehat{\mathbf{g}}^n(\alpha^n)}
 =& V\frac{\sum_{j=1}^mp_jy_j}{\sum_{j=1}^mp_jT_j}
 +\dotp{\mathbf{Q}[t^n_k]}{\frac{\sum_{j=1}^mp_j\mathbf{z}_j}{\sum_{j=1}^mp_jT_j}}\\
 =& \sum_{i=1}^m\frac{p_iT_i}{\sum_{j=1}^mp_jT_j}\left(V\frac{y_i}{T_i}+\dotp{\mathbf{Q}[t^n_k]}{\frac{\mathbf{z}_i}{T_i}}\right)\\
 =:& \sum_{i=1}^m q_i\left(V\frac{y_i}{T_i}+\dotp{\mathbf{Q}[t^n_k]}{\frac{\mathbf{z}_i}{T_i}}\right),
 \end{align*}
 where we let $q_i=\frac{p_iT_i}{\sum_{j=1}^mp_jT_j}$. 
Note that
 $q_i\geq0$ and $\sum_{i=1}^mq_i = 1$ because $T_i\geq1$. Hence, solving \eqref{DPP-ratio-simple} is equivalent to choosing $\{q_i\}_{i=1}^m$ to minimize the above expression, which boils down to choosing a single 
 $(y_i,\mathbf{z}_i,T_i)$ among $\{(y_j,\mathbf{z}_j,T_j)\}_{j=1}^m$ which achieves the minimum.

Note that such a convex hull case stands out not only because it yields a simple solution, but also because of the fact that ergodic coupled MDPs discussed in 
Section \ref{sec:MDP} have the region  $\left\{\left(\widehat{y}^n(\alpha^n),~\widehat{\mathbf{z}}^n(\alpha^n),~\widehat{T}^n(\alpha^n)\right):~\alpha^n\in\mathcal{A}^n\right\}$ 
being the convex hull of a finite sequence of points $\{(y_j,\mathbf{z}_j,T_j)\}_{j=1}^m$, where each point $(y_j,\mathbf{z}_j,T_j)$ results from a
 pure stationary policy (\cite{Al99}).
\footnote{A pure stationary policy is an algorithm where the decision to be taken at any time $t$ is a deterministic function of the state at time $t$, and independent of all other past information.} Thus, solving \eqref{DPP-ratio-simple} for the ergodic coupled MDPs reduces to choosing a pure policy among a finite number of pure policies.

\section{Limiting Performance}\label{section:limiting}
For the rest of the paper, the underlying probability space is denoted as the tuple $(\Omega,~\mathcal{F},~P)$.
 Let $\mathcal{F}[t]$ be the system history up until time slot $t$. Formally, $\{\mathcal{F}[t]\}_{t=0}^\infty$ is a filtration with $\mathcal{F}[0]=\{\emptyset,\Omega\}$ and each $\mathcal{F}[t],~t\geq1$ is the $\sigma$-algebra generated by all random variables from slot 0 to $t$.
 
 For the rest of the paper, we always assume Assumptions \ref{feasible-assumption}-\ref{compact-assumption} hold without explicitly mentioning them.

 \subsection{Convexity}
 The following lemma demonstrates the convexity of $\mathcal{P}^n$ in Definition \ref{PR-def}.

\begin{lemma}\label{convex-lemma}
The performance region $\mathcal{P}^n$ specified in Definition \ref{PR-def} is convex for any $n\in\{1,2,\cdots,N\}$.
Furthermore, it is the convex hull of the set
$\left\{ \left(\widehat{f}^n(\alpha^n),~\widehat{\mathbf{g}}^n(\alpha^n)\right) : \alpha^n\in\mathcal{A}^n \right\}$ and thus compact, where $\left(\widehat{f}^n(\alpha^n),~\widehat{\mathbf{g}}^n(\alpha^n)\right)$ is specified Definition \ref{PV-def}.
\end{lemma}
\begin{proof}
We first prove the convexity of $\mathcal{P}^n$.
Consider any two points $(f_1,\mathbf{g}_1),~(f_2,\mathbf{g}_2)\in\mathcal{P}^n$. We aim to show that for any $q\in(0,1)$, $(qf_1+(1-q)f_2,q\mathbf{g}_1+(1-q)\mathbf{g}_2)\in\mathcal{P}^n$. Notice that by definition of $\mathcal{P}^n$, there exists $(y_1,\mathbf{z}_1,T_1),~(y_2,\mathbf{z}_2,T_2)\in\mathcal{S}^n$ such that $f_1=y_1/T_1$, $\mathbf{g}_1=\mathbf{z}_1/T_1$, $f_2=y_2/T_2$, and $\mathbf{g}_2=\mathbf{z}_2/T_2$. Thus, it is enough to show 
\begin{equation}\label{convex-combo}
\left(q\frac{y_1}{T_1}+(1-q)\frac{y_2}{T_2},q\frac{\mathbf{z}_1}{T_1}+(1-q)\frac{\mathbf{z}_2}{T_2}\right)\in\mathcal{P}^n.
\end{equation}
To show this, we make a change of variable by letting $p=\frac{qT_2}{(1-q)T_1+qT_2}$. It is obvious that $p\in(0,1)$. Furthermore, $q=\frac{pT_1}{pT_1+(1-p)T_2}$ and 
\begin{align*}
&q\frac{y_1}{T_1}+(1-q)\frac{y_2}{T_2}=\frac{py_1+(1-p)y_2}{pT_1+(1-p)T_2},\\
&q\frac{\mathbf{z}_1}{T_1}+(1-q)\frac{\mathbf{z}_2}{T_2}
=\frac{p\mathbf{z}_1+(1-p)\mathbf{z}_2}{pT_1+(1-p)T_2}.
\end{align*}
Since $\mathcal{S}^n$ is convex, 
$$(py_1+(1-p)y_2,~p\mathbf{z}_1+(1-p)\mathbf{z}_2
,~pT_1+(1-p)T_2)\in\mathcal{S}^n.$$
Thus, by definition of $\mathcal{P}^n$ again, \eqref{convex-combo} holds and the first part of the proof is finished.

To show the second part of the claim, let 
$$\mathcal{Q}^n: = \left\{ \left(\widehat{f}^n(\alpha^n),~\widehat{\mathbf{g}}^n(\alpha^n)\right) : \alpha^n\in\mathcal{A}^n \right\}
= \left\{ \left(\widehat{y}^n(\alpha^n)\left/\widehat{T}^n(\alpha^n)\right.,~\widehat{\mathbf{z}}^n(\alpha^n)\left/\widehat{T}^n(\alpha^n)\right)\right. : \alpha^n\in\mathcal{A}^n \right\}$$ 
and let 
$\text{conv}(\mathcal{Q}^n)$ be the convex hull of $\mathcal{Q}^n$. First of all, By Definition \ref{PR-def},
\[\mathcal{P}^n=\left\{\left(y/T,~\mathbf{z}/T\right):~(y,\mathbf{z},T)\in\mathcal{S}^n\right\}\subseteq\mathbb{R}^{L+1},\]
for $\mathcal{S}^n$ being the convex hull of $\left\{\left(\widehat{y}^n(\alpha^n),~\widehat{\mathbf{z}}^n(\alpha^n),~\widehat{T}^n(\alpha^n)\right):~\alpha^n\in\mathcal{A}^n\right\}$, thus, in view of the definition of $\mathcal{Q}^n$, we have $\mathcal{Q}^n\subseteq\mathcal{P}^n$.
Since both 
$\mathcal{P}^n$ and $\text{conv}(\mathcal{Q}^n)$ are convex, by definition of convex hull (\cite{rockafellar2015convex}) that $\text{conv}(\mathcal{Q}^n)$ is the smallest convex set containing $\mathcal{Q}^n$,
 we have 
$\text{conv}(\mathcal{Q}^n)\subseteq\mathcal{P}^n$.

To show the reverse inclusion $\mathcal{P}^n\subseteq\text{conv}(\mathcal{Q}^n)$, note that any point in 
$\mathcal{P}^n$ can be written in the form $\left( \frac{y}{T},\frac{\mathbf{z}}{T} \right)$, where
$(y,\mathbf{z},T)\in\mathcal{S}^n$. Since $\mathcal{S}^n$ by definition is the convex hull of 
$$\left\{\left(\widehat{y}^n(\alpha^n),~\widehat{\mathbf{z}}^n(\alpha^n),~\widehat{T}^n(\alpha^n)\right):~\alpha^n\in\mathcal{A}^n\right\}\subseteq\mathbb{R}^{L+2},$$
by the definition of convex hull, $(y,\mathbf{z},T)$ can be written as a convex combination of
$m$ points in the above set. Let 
$\left\{\left(\widehat{y}^n(\alpha^n_i),~\widehat{\mathbf{z}}^n(\alpha^n_i),~\widehat{T}^n(\alpha^n_i)\right)\right\}_{i=1}^m$ be these points, so that
\begin{align*}
&(y,\mathbf{z},T) = \sum_{i=1}^m p_i\cdot\left(\widehat{y}^n(\alpha^n_i),~\widehat{\mathbf{z}}^n(\alpha^n_i),~\widehat{T}^n(\alpha^n_i)\right),\\
&p_i\geq0,~~\sum_{i=1}^mp_i = 1.
\end{align*}
As a result, we have
\[
\left( \frac{y}{T},\frac{\mathbf{z}}{T} \right)
=\left( \frac{\sum_{i=1}^m p_iy^n(\alpha^n_i)}{\sum_{i=1}^m p_iT^n(\alpha^n_i)},\frac{\sum_{i=1}^m p_i\mathbf{z}^n(\alpha^n_i)}{\sum_{i=1}^m p_iT^n(\alpha^n_i)} \right).
\]
We make a change of variable by letting $q_j = \frac{p_jT^n(\alpha^n_j)}{\sum_{i=1}^m p_iT^n(\alpha^n_i)},~\forall j=1,2,\cdots,m$, then, 
$$p_j = \frac{q_j}{T^n(\alpha^n_j)}\cdot\sum_{i=1}^m p_iT^n(\alpha^n_i),$$
it follows,
\[
\left( \frac{y}{T},\frac{\mathbf{z}}{T} \right)
=\sum_{i=1}^m q_i\cdot\left( \frac{ y^n(\alpha^n_i)}{T^n(\alpha^n_i)},\frac{\mathbf{z}^n(\alpha^n_i)}{T^n(\alpha^n_i)} \right) = \sum_{i=1}^m q_i\cdot\left(\widehat{f}^n(\alpha^n_i),~\widehat{\mathbf{g}}^n(\alpha^n_i)\right).
\]
Since $\sum_{i=1}^mq_i = 1$ and $q_i\geq0$, it follows any point in $\mathcal{P}^n$ can be written as a convex combination of finite number of points in $ \mathcal{Q}^n$, which implies 
$\mathcal{P}^n\subseteq\text{conv}(\mathcal{Q}^n)$. Overall, we have $\mathcal{P}^n=\text{conv}(\mathcal{Q}^n)$. 

Finally, by Assumption \ref{compact-assumption}, we have $\mathcal{Q}^n=\left\{ \left(\widehat{f}^n(\alpha^n),~\widehat{\mathbf{g}}^n(\alpha^n)\right) : \alpha^n\in\mathcal{A}^n \right\}$ is compact. Thus, $\mathcal{P}^n$, being a convex hull of a compact set, is also compact.
\end{proof}

\subsection{Key-feature inequality and supermartingale construction}\label{sec-4.2}
First of all, we have the following fundamental performance lemma which states that the optimality of \eqref{prob-1}-\eqref{prob-2} is achievable within $\mathcal{P}^n$ specified in Definition \ref{PR-def}.
\begin{lemma}\label{stationary-lemma}
For each $n\in\{1,2,\cdots,N\}$, there exists a pair 
$\left(\overline{f}^n_*,~\overline{\mathbf{g}}^n_*\right)\in\mathcal{P}^n$ such that the following hold:
\begin{align*}
&\sum_{n=1}^N\overline{f}^n_*=f_*\\
&\sum_{n=1}^N\overline{g}^n_{l,*}\leq d_l,~l\in\{1,2,\cdots,L\},
\end{align*}
where $f^*$ is the optimal objective value for problem \eqref{prob-1}-\eqref{prob-2}, i.e. the optimality is achievable within $\otimes_{n=1}^N\mathcal{P}^n$, the Cartesian product of $\mathcal{P}^n$. 

Furthermore, for any $\left(\overline{f}^n,~\overline{\mathbf{g}}^n\right)\in\mathcal{P}^n,~n\in\{1,2,\cdots, N\}$, satisfying 
$\sum_{n=1}^N\overline{g}^n_{l}\leq d_l,~l\in\{1,2,\cdots,L\}$, we have $\sum_{n=1}^N\overline{f}^n\geq f_*$, i.e. one cannot achieve better performance than \eqref{prob-1}-\eqref{prob-2} in $\otimes_{n=1}^N\mathcal{P}^n$.
\end{lemma}
The proof of this Lemma is delayed to Appendix \ref{appendix-proof}. In particular, the proof uses the following lemma, which also plays an important role in several lemmas later.

\begin{lemma}\label{bound-lemma-1}
Suppose $\{y^n[t]\}_{t=0}^\infty$, $\{\mathbf{z}^n[t]\}_{t=0}^\infty$ and $\{T^n_k\}_{k=0}^\infty$ are processes resulting from any algorithm,\footnote{Note that this algorithm might make decisions using the past information.} then, $\forall T\in\mathbb{N}$,
\begin{align}
&\frac1T\sum_{t=0}^{T-1}\expect{f^n[t]-y^n[t]}\leq\frac{B_1}{T},\label{bound-1}\\
&\frac1T\sum_{t=0}^{T-1}\expect{g^n_l[t]-z^n_l[t]}\leq\frac{B_2}{T},~l\in\{1,2,\cdots,L\},
\label{bound-2}
\end{align}
where $B_1=2y_{\max}\sqrt{B}$, $B_2=2z_{\max}\sqrt{B}$ and $f^n[t]$, $\mathbf{g}^n[t]$ are constant over each renewal frame for system $n$ defined by
\begin{align*}
f^n[t]=\widehat{f}^n(\alpha^n),~~\textrm{if}~t\in\mathcal{T}^n_k,\alpha^n_k=\alpha^n\\
\mathbf{g}^n[t]=\widehat{\mathbf{g}}^n(\alpha^n),~~\textrm{if}~t\in\mathcal{T}^n_k,
\alpha^n_k=\alpha^n,
\end{align*}
and $\left(\widehat{f}^n(\alpha^n),\widehat{\mathbf{g}}^n(\alpha^n)\right)$ are defined in Definition \ref{PV-def}.
\end{lemma}
The proof of this lemma is delayed to Appendix \ref{appendix-proof}.

\begin{remark}
Note that directly computing $\overline{f}^n_*$ and $\overline{g}^n_{l,*}$ indicated by Lemma \ref{stationary-lemma} would be difficult because of the fractional nature of $\mathcal{P}^n$, the coupling between different systems through time average constraints and the fact that $d_l=\expect{d_l[t]}$ might be unknown. However, Lemma \ref{stationary-lemma} can be used to prove important performance theorems regarding our proposed algorithm as is indicated by the following lemma.
\end{remark}

The following key-feature inequality connects our proposed algorithm with the performance vectors inside $\mathcal{P}^n$.

\begin{lemma}\label{key-feature}
Consider the stochastic processes $\{y^n[t]\}_{t=0}^\infty$, $\{\mathbf{z}^n[t]\}_{t=0}^\infty$, and 
$\{T^n_k\}_{k=0}^\infty$ resulting from the proposed algorithm. For any system $n$, the following holds for any $k\in\mathbb{N}$ and any 
$(\overline{f}^n,\overline{\mathbf{g}}^n)\in\mathcal{P}^n$,
\begin{align}\label{key-feature-in}
\frac{\expect{\left.\sum_{t\in\mathcal{T}^n_k}\left(Vy^n[t]+\dotp{\mathbf{Q}[t^n_k]}{\mathbf{z}^n[t]}\right)\right|\mathbf{Q}[t^n_k]}}{\expect{T^n_k|\mathbf{Q}[t^n_k]}}\leq V\overline{f}^n+\dotp{\mathbf{Q}[t^n_k]}{\overline{\mathbf{g}}^n},
\end{align}
\end{lemma}

\begin{proof}
First of all,
since the proposed algorithm solves \eqref{DPP-ratio} over all possible decisions in $\mathcal{A}^n$, it must achieve value less than or equal to that of any action $\alpha^n\in\mathcal{A}^n$ at the same frame. This gives,
\begin{align*}
D^n_k\leq\frac{\expect{\left.\sum_{t\in\mathcal{T}^n_k}\left(Vy^n[t]+\dotp{\mathbf{Q}[t^n_k]}{\mathbf{z}^n[t]}\right)\right|\mathbf{Q}[t^n_k],\alpha^n_k=\alpha^n}}{\expect{\left.T^n_k\right|\mathbf{Q}[t^n_k],\alpha^n_k=\alpha^n}}
=\frac{V\widehat{y}^n(\alpha^n)+\dotp{\mathbf{Q}[t^n_k]}{\widehat{\mathbf{z}}^n(\alpha^n)}}{\widehat{T}^n(\alpha^n)},
\end{align*}
where $D^n_k$ is defined in \eqref{DPP-ratio} and 
the equality follows from the renewal property of the system that $T^n_k$, $\sum_{t\in\mathcal{T}^n_k}y^n[t]$ and $\sum_{t\in\mathcal{T}^n_k}\mathbf{z}^n[t]$ are conditionally independent of $\mathbf{Q}[t^n_k]$ given $\alpha^n_k=\alpha^n$.

Since $T^n_k\geq1$, this implies
\begin{align*}
\widehat{T}^n(\alpha^n)\cdot D^n_k\leq V\widehat{y}^n(\alpha^n)+\dotp{\mathbf{Q}[t^n_k]}{\widehat{\mathbf{z}}^n(\alpha^n)},
\end{align*}
thus, for any $\alpha^n\in\mathcal{A}^n$,
\[V\widehat{y}^n(\alpha^n)+\dotp{\mathbf{Q}[t^n_k]}{\widehat{\mathbf{z}}^n(\alpha^n)}
-D^n_k\cdot\widehat{T}^n(\alpha^n)\geq0.\]
Since $\mathcal{S}^n$ specified in Definition \ref{PR-def} is the convex hull of 
$\left\{(\widehat{y}^n(\alpha^n),~\widehat{\mathbf{z}}^n(\alpha^n),~\widehat{T}^n(\alpha^n)),~\alpha^n\in\mathcal{A}^n\right\}$,
it follows for any vector $(y,\mathbf{z},T)\in\mathcal{S}^n$, we have
\[Vy+\dotp{\mathbf{Q}[t^n_k]}{\mathbf{z}}
-D^n_k\cdot T\geq0.\]
Dividing both sides by $T$ and using the definition of $\mathcal{P}^n$ in Definition \ref{PR-def} give
\[D^n_k\leq V\overline{f}^n+\dotp{\mathbf{Q}[t^n_k]}{\overline{\mathbf{g}}^n},~\forall(\overline{f}^n,\overline{\mathbf{g}}^n)\in\mathcal{P}^n.\]
Finally, since $\{y^n[t]\}_{t=0}^\infty$, $\{\mathbf{z}^n[t]\}_{t=0}^\infty$, and 
$\{T^n_k\}_{k=0}^\infty$ result from the proposed algorithm and the action chosen is determined by $\mathbf{Q}[t^n_k]$ as in \eqref{DPP-ratio},
\[D^n_k=\frac{\expect{\left.\sum_{t\in\mathcal{T}^n_k}\left(Vy^n[t]+\dotp{\mathbf{Q}[t^n_k]}{\mathbf{z}^n[t]}\right)\right|\mathbf{Q}[t^n_k]}}{\expect{T^n_k|\mathbf{Q}[t^n_k]}}.\]
This finishes the proof.
\end{proof}

Our next step is to give a frame-based analysis for each system by constructing a supermartingale on the per-frame timescale. Recall that 
$\{\mathcal{F}[t]\}_{t=0}^{\infty}$ is a filtration (with $\mathcal F[t]$ representing system history during slots $\{0, \cdots, t\}$). Fix a system $n$ and  recall that $t_k^n$ is the time slot where the $k$-th renewal occurs for system $n$.  We would like to define a filtration corresponding to the random times 
$t_k^n$. To this end, define the collection of sets $\{\mathcal F_k^n\}_{k=0}^\infty$ such that for each $k$, 
\[
\mathcal F_k^n := \{A \in\mathcal F : A \cap \{t_k^n \leq t\} \in \mathcal F[t], \forall t \in \{0, 1, 2,\cdots\}\} 
\]

For example, the following set $A$ is an element of $\mathcal F_3^n$: 
\[
A = \{t_3^n=5\} \cap \{y[0]=y_0, y[1]=y_1, y[2]=y_2, y[3]=y_3, y[4]=y_4\}
\]
where $y_0,\cdots, y_4$ are specific values. Then $A \in\mathcal F_3^n$ because 
for $i \in \{0, 1, 2, 3, 4\}$ we have $A \cap \{t_3^n \leq i\} = \emptyset \in \mathcal F[i]$, and for $i \in \{5, 6, 7, \cdots\}$ we have $A \cap \{t_3^n \leq i\} = A \in \mathcal F[i]$.  The following technical lemma is proved in the appendix.

\begin{lemma}\label{lemma:filtration}
The sequence $\{\mathcal F_k^n \}_{k=0}^\infty$ is a valid filtration, i.e. 
$\mathcal F_k^n \subseteq\mathcal F_{k+1}^n ,~\forall k\geq0$. 
Furthermore, for any real-valued adapted process $\{Z^n[t-1]\}_{t=1}^\infty$ with respect to $\{\mathcal{F}[t]\}_{t=1}^\infty$,
\footnote{Meaning that for each $t$ in $\{1, 2,3, \cdots\}$, the random variable $Z^n[t-1]$ is determined by events in $\mathcal F[t]$.}  
$$\left\{G_{t^n_k}(Z^n[0],~Z^n[1],~\cdots,Z^n[t^n_k-1])\right\}_{k=1}^\infty$$ 
is also adapted to $\{\mathcal F_k^n \}_{k=1}^\infty$, where for any $t\in\mathbb{N}$, $G_t(\cdot)$ is a fixed real-valued measurable mappings.
That is, for any $k$, it holds that any measurable 
function of $(Z^n[0], \cdots, Z[t_k^n-1])$ is determined by events in $\mathcal F_k^n$. 
\end{lemma}

With Lemma \ref{key-feature} and Lemma \ref{lemma:filtration}, we can construct a supermartingale as follows,

\begin{lemma}\label{supMG}
Consider the stochastic processes $\{y^n[t]\}_{t=0}^\infty$, $\{\mathbf{z}^n[t]\}_{t=0}^\infty$, and 
$\{T^n_k\}_{k=0}^\infty$ resulting from the proposed algorithm. For any $(\overline{f}^n,\overline{\mathbf{g}}^n)\in\mathcal{P}^n$,
let
\begin{equation} \label{def-X}
X^n[t]:=V\left(y^n[t]-\overline{f}^n\right)+\dotp{\mathbf{Q}[t]}{\mathbf{z}^n[t]-\overline{\mathbf{g}}^n},
\end{equation}
then,
\[\expect{\left.\sum_{t\in\mathcal{T}^n_k}X^n[t]\right|\mathcal F_k^n }\leq Lz_{\max}(Nz_{\max}+d_{\max})B:=C_0,\]
where $B$, $z_{\max}$ and $d_{\max}$ are as defined in Assumption \ref{bounded-assumption}. Furthermore, define a real-valued process $\{Y^n_K\}_{K=0}^\infty$ on the frame such that $Y^n_0=0$ and
\[Y^n_K=\sum_{k=0}^{K-1}\left(\sum_{t\in\mathcal{T}^n_k}X^n[t]-C_0\right),~K\geq1.\]
Then, $\{Y^n_K\}_{K=0}^\infty$ is a supermartingale adapted to the aforementioned filtration 
$\{\mathcal F_K^n \}_{K=0}^\infty$.
\end{lemma}
\begin{proof}
Consider any $t\in\mathcal{T}^n_k$, then, we can decompose $X^n[t]$ as follows
\begin{align}\label{eq-decompose} 
X^n[t]=&V(y^n[t]-\overline{f}^n)+\dotp{\mathbf{Q}[t^n_k]}{\mathbf{z}^n[t]-\overline{\mathbf{g}}^n}
+\dotp{\mathbf{Q}[t]-\mathbf{Q}[t^n_k]}{\mathbf{z}^n[t]-\overline{\mathbf{g}}^n}.
\end{align}
By the queue updating rule \eqref{queue-update}, we have for any $l\in\{1,2,\cdots,L\}$ and any $t>t^n_k$,
\begin{equation}\label{queue-update-bound}
|Q_l[t]-Q_l[t^n_k]|\leq\sum_{s=t^n_k}^{t-1}\left|\sum_{m=1}^Nz^m_l[s]-d_l[t]\right|
\leq(t-t^n_k)(Nz_{\max}+d_{\max})
\end{equation}
Thus, for the last term in \eqref{eq-decompose}, by H\"{o}lder's inequality, 
\begin{align*}
\dotp{\mathbf{Q}[t]-\mathbf{Q}[t^n_k]}{\mathbf{z}^n[t]-\overline{\mathbf{g}}^n}
\leq&\|\mathbf{Q}[t]-\mathbf{Q}[t^n_k]\|_1\cdot\|\mathbf{z}^n[t]-\overline{\mathbf{g}}^n\|_{\infty}\\
\leq&\sum_{s=t^n_k}^{t-1}\left\|\sum_{m=1}^N\mathbf{z}^n[s]-\mathbf{d}[t]\right\|_1\cdot\|\mathbf{z}^n[t]-\overline{\mathbf{g}}^n\|_{\infty}\\
\leq&(t-t^n_k)L(Nz_{\max}+d_{\max})\cdot2z_{\max},
\end{align*}
where the second inequality follows from  \eqref{queue-update-bound} and the last inequality follows from the boundedness assumption (Assumption \ref{bounded-assumption}) of corresponding quantities. Substituting the above bound into \eqref{eq-decompose} gives a bound on $\expect{\left.\sum_{t\in\mathcal{T}^n_k}X^n[t]\right|\mathcal F_k^n }$ as
\begin{align}
\expect{\left.\sum_{t\in\mathcal{T}^n_k}X^n[t]\right|\mathcal F_k^n }
\leq&\expect{\left.\sum_{t\in\mathcal{T}^n_k}\left(V\left(y^n[t]-\overline{f}^n\right)+\dotp{\mathbf{Q}[t^n_k]}{\mathbf{z}^n[t]-\overline{\mathbf{g}}^n}\right)\right|\mathcal F_k^n }\nonumber\\
&+\expect{\left.\sum_{t\in\mathcal{T}^n_k}(t-t^n_k)\right|\mathcal F_k^n }
\cdot2L(Nz_{\max}+d_{\max})z_{\max}\nonumber\\
\leq&\expect{\left.\sum_{t\in\mathcal{T}^n_k}\left(V\left(y^n[t]-\overline{f}^n\right)+\dotp{\mathbf{Q}[t^n_k]}{\mathbf{z}^n[t]-\overline{\mathbf{g}}^n}\right)\right|\mathcal F_k^n }\nonumber\\
&+\expect{\left.(T^n_k)^2\right|\mathcal F_k^n }
\cdot L(Nz_{\max}+d_{\max})z_{\max},\label{bound-on-X}
\end{align}
where we use the fact that $0+1+\cdots+T^n_k-1 = (T^n_k-1)T^n_k/2\leq (T^n_k)^2$ in the last inequality.

Next, by the queue updating rule \eqref{queue-update}, $Q_l[t^n_k]$ is determined by $z_l^n[0],\cdots,z_l^n[t^n_k-1]$
($n=1,2,\cdots,N$) and $d_l[0],\cdots,d_l[t^n_k-1]$ for any $l\in\{1,2,\cdots,L\}$. Thus, by Lemma \ref{lemma:filtration},
$\mathbf{Q}[t^n_k]$ is determined by $\mathcal F_k^n $. 
For the proposed algorithm, each system makes decisions purely based on the virtual queue state $\mathbf{Q}[t^n_k]$, and
by the renewal property of each system, given the decision at the $k$-th renewal, the random quantities $T^n_k$, $\mathbf{z}^n[t]$ and $y^n[t]$,~$t\in\mathcal{T}^n_k$ are independent of the outcomes from the slots before $t^n_k$. 
This implies the following display,
\begin{align}
&\expect{\left.\sum_{t\in\mathcal{T}^n_k}\left(V\left(y^n[t]-\overline{f}^n\right)+\dotp{\mathbf{Q}[t^n_k]}{\mathbf{z}^n[t]-\overline{\mathbf{g}}^n}\right)\right|\mathcal F_k^n }\nonumber\\
&=\expect{\left.\sum_{t\in\mathcal{T}^n_k}V\left(y^n[t]-\overline{f}^n\right)\right|\mathcal F_k^n }+\dotp{\mathbf{Q}[t^n_k]}{\expect{\left.\sum_{t\in\mathcal{T}^n_k}\left(\mathbf{z}^n[t]-\overline{\mathbf{g}}^n\right)\right|~\mathcal F_k^n }}\nonumber\\
&=\expect{\left.\sum_{t\in\mathcal{T}^n_k}V\left(y^n[t]-\overline{f}^n\right)\right|\mathbf{Q}[t^n_k]}+\dotp{\mathbf{Q}[t^n_k]}{\expect{\left.\sum_{t\in\mathcal{T}^n_k}\left(\mathbf{z}^n[t]-\overline{\mathbf{g}}^n\right)\right|~\mathbf{Q}[t^n_k]}}\nonumber\\
&=\expect{\left.\sum_{t\in\mathcal{T}^n_k}\left(V\left(y^n[t]-\overline{f}^n\right)+\dotp{\mathbf{Q}[t^n_k]}{\mathbf{z}^n[t]-\overline{\mathbf{g}}^n}\right)\right|\mathbf{Q}[t^n_k]},\label{mark-1}
\end{align}
By Lemma \ref{key-feature}, we have the following:
\begin{align*}
\expect{\left.\sum_{t\in\mathcal{T}^n_k}\left(Vy^n[t]+\dotp{\mathbf{Q}[t^n_k]}{\mathbf{z}^n[t]}\right)\right|\mathbf{Q}[t^n_k]}
\leq \left(V\overline{f}^n+\dotp{\mathbf{Q}[t^n_k]}{\overline{\mathbf{g}}^n}\right)\cdot\expect{T^n_k|\mathbf{Q}[t^n_k]}.
\end{align*}
Thus, rearranging terms in above inequality gives
the expectation on the right hand side of \eqref{mark-1} is no greater than 0 and hence the first expectation on the right hand side of \eqref{bound-on-X} is also no greater than 0. For the second expectation in \eqref{bound-on-X}, using \eqref{residual-life-bound} in Assumption \ref{bounded-assumption} gives $\expect{\left.(T^n_k)^2\right|\mathcal F_k^n }\leq B$ and the first part of the lemma is proved. 

For the second part of the lemma, by Lemma \ref{lemma:filtration} and the definition of $Y^n_K$, the process $\{Y^n_K\}_{K=0}^\infty$
is adapted to $\{\mathcal F_k^n \}_{K=0}^{\infty}$.
Moreover, by Assumption \ref{bounded-assumption},
\begin{align*}
\expect{\left|\sum_{t\in\mathcal{T}^n_k}X^n[t]\right|}
\leq\expect{\sum_{t\in\mathcal{T}^n_k}\left|X^n[t]\right|}<\infty,~\forall k.
\end{align*}
Thus, $\expect{|Y^n_K|}<\infty,~\forall K\in\mathbb{N}$, i.e. it is absolutely integrable. Furthermore, by the first part of the lemma,
\begin{align*}
\expect{Y^n_{K+1}~|~\mathcal F_k^n }
=Y^n_K+\expect{\left.\left(\sum_{t\in\mathcal{T}^n_K}X^n[t]-C_0\right)~\right|~\mathcal F_k^n }
\leq Y^n_K,
\end{align*}
finishing the proof.
\end{proof}

\subsection{Synchronization lemma}\label{section:sync}
So far, we have analyzed the processes related to each individual system over its renewal frames. However, due the asynchronous behavior of different systems, the supermartingales of each system cannot be immediately summed.

In order to get a global performance bound, we have to get rid of any index related to individual renewal frames only. In other words, 
we need to look at the system property at any time slot $T$ as opposed to any renewal $t^n_k$. 
We start with the following standard definition of stopping time:

 \begin{definition}
Given a probability space $(\Omega, \mathcal{F}, P)$ and a filtration
$\{\varnothing, \Omega\}=\mathcal{F}_0\subseteq\mathcal{F}_1\subseteq\mathcal{F}_2\cdots$
in $\mathcal{F}$. A stopping time $\tau$ with respect to the filtration $\{\mathcal{F}_i\}_{i=0}^{\infty}$ is a random variable such that for any $i\in\mathbb{N}$,
\[\{\tau=i\}\in\mathcal{F}_i,\]
i.e. the stopping time occurring at time $i$ is contained in the information during slots $0,~1,~2,~\cdots,~i-1$.
\end{definition}

Next, for any fixed slot $T>0$, let $S^n[T]$ be the number of renewals up to (and including) time slot $T$, with with the convention that the first renewal occurs at time $t=0$, so $t_0^n=0$ and $S^n[0]=1$.
The next lemma shows $S^n[T]$ is a valid stopping time, whose proof is in the appendix.

\begin{lemma}\label{valid-stopping-time}
For each $n\in\{1,2,\cdots,N\}$, 
the random variable $S^n[T]$ is a stopping time
with respect to the filtration $\{\mathcal F_k^n \}_{k=0}^\infty$, i.e. $\{S^n[T]= k\}\in\mathcal F_k^n ,~\forall k\in\mathbb{N}$.
\end{lemma}

The following theorem tells us a stopping-time truncated supermartingale is still a supermartingale.
\begin{theorem}[Theorem 5.2.6 in \cite{Durrett}]\label{stopping-time}
If $\tau$ is a stopping time and $Z[i]$ is a supermartingale with respect to $\{\mathcal{F}_i\}_{i=0}^\infty$, then $Z[i\wedge \tau]$ is also a supermartingale, where $a\wedge b\triangleq\min\{a,b\}$.
\end{theorem}

With this theorem and the above stopping time construction, we have the following lemma:

\begin{lemma}\label{sync-lemma}
For each $n\in\{1,2,\cdots,N\}$ and any fixed $T\in\mathbb{N}$, we have
\begin{align*}
\frac1T\sum_{t=0}^{T-1}\expect{X^n[t]}\leq C_1+\frac{C_2V}{T},
\end{align*}
where 
\[C_1:=6Lz_{\max}(Nz_{\max}+d_{\max})B,~~C_2:=2y_{\max}\sqrt{B}.\]
\end{lemma}
\begin{proof}
First, note that the renewal index $k$ starts from 0. Thus,
for any fixed $T\in \mathbb{N}$, $ t^n_{S^n[T]-1}\leq T<t^n_{S^n[T]}$, and
\begin{align}
\expect{\sum_{t=0}^{T-1}X^n[t]}=&\expect{\sum_{t=0}^{t^n_{S^n[T]}-1}X^n[t]-\sum_{t=T}^{t^n_{S^n[T]}-1}X^n[t]}\nonumber\\
=&\expect{\sum_{t=1}^{t^n_{S^n[T]}-1}X^n[t]}-\expect{\sum_{t=T}^{t^n_{S^n[T]}-1}X^n[t]}\nonumber\\
=&\expect{Y^n_{S^n[T]}}+C_0\expect{S^n[T]}-\expect{\sum_{t=T}^{t^n_{S^n[T]}-1}X^n[t]}\nonumber\\
\leq&\expect{Y^n_{S^n[T]}}+C_0(T+1)-\expect{\sum_{t=T}^{t^n_{S^n[T]}-1}X^n[t]},\label{decompose-inequality}
\end{align}
where the third equality follows from the definition of $Y^n_K$ in Lemma \ref{supMG} and the last inequality follows from the fact that the number of renewals up to time slot $T$ is no more than the total number of slots, i.e.
$S^n[T]\leq T+1$. For the term $\expect{Y^n_{S^n[T]}}$, we apply Theorem \ref{stopping-time} with 
$\tau = S^n[T]$ and index $K$ to obtain
$\{Y^n_{K\wedge S^n[T]}\}_{K=0}^\infty$ is a supermartingale. This implies
\[\expect{Y^n_{K\wedge S^n[T]}}\leq\expect{Y^n_{0\wedge S^n[T]}}=\expect{Y^n_0}=0,~\forall K\in\mathbb{N}.\]
Since $S^n[T]\leq T+1$, it follows by substituting $K=T+1$,
\[\expect{Y^n_{S^n[T]}}=\expect{Y^n_{(T+1)\wedge S^n[T]}}\leq0.\]
For the last term in \eqref{decompose-inequality}, by queue updating rule \eqref{queue-update}, for any $l\in\{1,2,\cdots,L\}$,
\[|Q_l[t]|\leq\sum_{s=0}^{t-1}\left|\sum_{m=1}^Nz^m_l[s]-d_l[t]\right|
\leq t(Nz_{\max}+d_{\max}),\]
it then follows from H\"{o}lder's inequality again that
\begin{align*}
\expect{\left|\sum_{t=T}^{t^n_{S^n[T]}-1}X^n[t]\right|}
=&\expect{\left|\sum_{t=T}^{t^n_{S^n[T]}-1}\left(V(y^n[t]-\overline{f}^n)+\dotp{\mathbf{Q}[t]}{\mathbf{z}^n[t]-\overline{\mathbf{g}}^n}\right)\right|}\\
\leq&\expect{\sum_{t=T}^{t^n_{S^n[T]}-1}\left(V\left|y^n[t]-\overline{f}^n\right|+\|\mathbf{Q}[t]\|_1\cdot\|\mathbf{z}^n[t]-\overline{\mathbf{g}}^n\|_{\infty}\right)}\\
\leq&\expect{\sum_{t=T}^{t^n_{S^n[T]}-1}\left(2Vy_{\max}+L(Nz_{\max}+d_{\max})t\cdot2z_{\max}\right)}\\
=&2Vy_{\max}\cdot\expect{t^n_{S^n[T]}-T}+Lz_{\max}(Nz_{\max}+d_{\max})\\
&\cdot\left((2T-1)\cdot\expect{t^n_{S^n[T]}-T}+\expect{t^n_{S^n[T]}-T}^2\right)\\
\leq& 2Vy_{\max}\sqrt{B}+2Lz_{\max}(Nz_{\max}+d_{\max})\sqrt{B}T+Lz_{\max}(Nz_{\max}+d_{\max})B\\
\leq& 2Vy_{\max}\sqrt{B}+2Lz_{\max}(Nz_{\max}+d_{\max})B(T+1),
\end{align*} 
where in the second from last inequality we use \eqref{residual-life-bound} of Assumption \ref{bounded-assumption} that the residual life $t^n_{S^n[T]}-T$ satisfies 
$$\expect{(t^n_{S^n[T]}-T)^2}
=\expect{\expect{\left.(t^n_{S^n[T]}-T)^2\right|~t^n_{S^n[T]}-t^n_{S^n[T]-1}\geq T-t^n_{S^n[T]-1}}}\leq B$$ 
and $\expect{t^n_{S^n[T]}-T}\leq\sqrt{B}$, and in the last inequality we use the fact that $B\geq1$, thus, $\sqrt{B}\leq B$.
Substitute the above bound into \eqref{decompose-inequality} gives
 \begin{align*}
 \expect{\sum_{t=0}^{T-1}X^n[t]}\leq& C_0(T+1)+2Vy_{\max}B+2Lz_{\max}(Nz_{\max}+d_{\max})B(T+1)\\
 =&2Vy_{\max}\sqrt{B}+3Lz_{\max}(Nz_{\max}+d_{\max})B(T+1)\\
 \leq&2Vy_{\max}\sqrt{B}+6Lz_{\max}(z_{\max}+d_{\max})BT
 \end{align*}
 where we use the definition $C_0=Lz_{\max}(z_{\max}+d_{\max})B$ from Lemma \ref{supMG} in the equality and use $T+1\leq2T$ in the final equality.
 Dividing both sides by $T$ finishes the proof.
\end{proof}

\subsection{Achieving near optimality}\label{sec-4.4}
The following theorem gives the performance bound of our proposed algorithm.

\begin{theorem}
The sequences $\{y^n[t]\}_{t=0}^\infty$ and $\{\mathbf{z}^n[t]\}_{t=0}^\infty$ produced by the proposed algorithm satisfy all the constraints in \eqref{prob-2} and achieves $\mathcal{O}(1/V)$ near optimality, i.e.
\[\limsup_{T\rightarrow\infty}\frac1T\sum_{t=0}^{T-1}\sum_{n=1}^N\expect{y^n[t]}\leq f_*+\frac{NC_1+C_3}{V},\]
where $f_*$ is the optimal objective of \eqref{prob-1}-\eqref{prob-2}, $C_1$ is defined in Lemma \ref{sync-lemma} and $C_3:=(Nz_{\max}+d_{\max})^2L/2$.
\end{theorem}
\begin{proof}
Define the drift-plus-penalty expression at time $t$ as 
\begin{equation}\label{compound-dpp}
P[t]:=\expect{\sum_{n=1}^NVy^n[t]+\frac12\left(\|\mathbf{Q}[t+1]\|^2-\|\mathbf{Q}[t]\|^2\right)}.
\end{equation}
By the queue updating rule \eqref{queue-update}, we have
\begin{align*}
P[t]\leq&\expect{\sum_{n=1}^NVy^n[t]+\frac12\sum_{l=1}^L\left(\sum_{n=1}^Nz^n_l[t]-d_l[t]\right)^2
+\sum_{l=1}^LQ_l[t]\left(\sum_{n=1}^Nz^n_l[t]-d_l[t]\right)}\\
\leq&\frac12(Nz_{\max}+d_{\max})^2L+\expect{\sum_{n=1}^NVy^n[t]
+\sum_{l=1}^LQ_l[t]\left(\sum_{n=1}^Nz^n_l[t]-d_l[t]\right)}\\
=&\frac12(Nz_{\max}+d_{\max})^2L+\expect{\sum_{n=1}^NVy^n[t]
+\sum_{l=1}^LQ_l[t]\left(\sum_{n=1}^Nz^n_l[t]-d_l\right)}
\end{align*}
where the second inequality follows from the boundedness assumption (Assumption \ref{bounded-assumption}) that $\sum_{l=1}^L\left(\sum_{n=1}^Nz^n_l[t]-d_l[t]\right)^2\leq (Nz_{\max}+d_{\max})^2L$, and the equality follows from the fact that $d_l[t]$ is i.i.d. and independent of $Q_l[t]$, thus,
\[\expect{Q_l[t]d_l[t]}=\expect{Q_l[t]\cdot\expect{d_l[t]|Q_l[t]}}=\expect{Q_l[t]d_l}.\]
For simplicity, define $C_3=\frac12(Nz_{\max}+d_{\max})^2L$. Now, by the achievability of optimality in $\otimes_{n=1}^N\mathcal{P}^n$ (Lemma \ref{stationary-lemma}), we have $\sum_{n=1}^N\overline{g}^n_{l,*}\leq d_l$, thus, substituting this inequality into the above bound for $P[t]$ gives
\begin{align*}
P[t]\leq& C_3+\expect{\sum_{n=1}^NVy^n[t]
+\sum_{n=1}^N\sum_{l=1}^LQ_l[t]\left(z^n_l[t]-\overline{g}^n_{l,*}\right)}\\
=&C_3+\sum_{n=1}^N\expect{Vy^n[t]+\dotp{\mathbf{Q}[t]}{\mathbf{z}^n[t]-\overline{\mathbf{g}}^n_*}}\\
=&C_3+\sum_{n=1}^N\expect{X^n[t]}+V\sum_{n=1}^N\overline{f}^n_*\\
=&C_3+\sum_{n=1}^N\expect{X^n[t]}+Vf_*,
\end{align*}
where we use the definition of $X^n[t]$ in \eqref{def-X} by substituting $(\overline{f}^n,\overline{\mathbf{g}}^n)$ with $(\overline{f}^n_*,\overline{\mathbf{g}}^n_*)$, i.e.
 $X^n[t]=V(y^n[t]-\overline{f}^n_*)+\dotp{\mathbf{Q}[t]}{\mathbf{z}^n[t]-\overline{\mathbf{g}}^n_*}$, in the second from last equality and use the optimality condition (Lemma \ref{stationary-lemma}) in the final equality. Now, by Lemma \ref{sync-lemma}, we have for any $T\in\mathbb{N}$,
 \begin{align}
\frac1T\sum_{t=0}^{T-1}P[t]\leq&
C_3+Vf_*+\frac1T\sum_{t=0}^{T-1}\sum_{n=1}^N\expect{X^n[t]}\nonumber\\
=&C_3+Vf_*+\sum_{n=1}^N\frac1T\sum_{t=0}^{T-1}\expect{X^n[t]}\nonumber\\
\leq&NC_1+C_3+Vf_*+\frac{NC_2V}{T}.\label{inter-dpp}
\end{align}

On the other hand, by the definition of $P[t]$ in \eqref{compound-dpp} and then telescoping sums with $\mathbf{Q}[0]=0$, we have
\begin{align*}
\frac1T\sum_{t=0}^{T-1}P[t]
=&\frac1T\sum_{t=0}^{T-1}\expect{\sum_{n=1}^NVy^n[t]+\frac12\left(\|\mathbf{Q}[t+1]\|^2-\|\mathbf{Q}[t]\|^2\right)}\\
=&\frac1T\sum_{t=0}^{T-1}\sum_{n=1}^NV\expect{y^n[t]}
+\frac{1}{2T}\expect{\|\mathbf{Q}[T]\|^2}.
\end{align*}
Combining this with inequality \eqref{inter-dpp} gives
\begin{equation}\label{final-dpp}
\frac1T\sum_{t=0}^{T-1}\sum_{n=1}^NV\expect{y^n[t]}
+\frac{1}{2T}\expect{\|\mathbf{Q}[T]\|^2}
\leq NC_1+C_3+Vf_*+\frac{NC_2V}{T}.
\end{equation}
Since $\frac{1}{2T}\expect{\|\mathbf{Q}[T]\|_2^2}\geq0$, we can throw away the term and the inequality still holds, i.e. 
\begin{equation}\label{final-dpp-2}
\frac1T\sum_{t=0}^{T-1}\sum_{n=1}^N\expect{y^n[t]}
\leq f_*+\frac{NC_1+C_3}{V}+\frac{NC_2}{T}.
\end{equation}
Taking $\limsup_{T\rightarrow\infty}$ from both sides gives the near optimality in the theorem.

To get the constraint violation bound, we use Assumption \ref{bounded-assumption} that $|y^n[t]|\leq y_{\max}$, then, by \eqref{final-dpp} again, we have
\begin{align*}
\frac{1}{T}\expect{\|\mathbf{Q}[T]\|^2}
\leq 2(NC_1+C_3)+4Vy_{\max}+\frac{2NC_2V}{T}.
\end{align*}
By Jensen's inequality $\expect{\|\mathbf{Q}[T]\|^2}\geq\expect{\|\mathbf{Q}[T]\|}^2$.
This implies that 
\[
\expect{\|\mathbf{Q}[T]\|}\leq\sqrt{(2(NC_1+C_3)+4Vy_{\max})T+2NC_2V},
\]
which implies
\begin{equation}\label{key-bound}
\frac{1}{T}\expect{\|\mathbf{Q}[T]\|}\leq\sqrt{\frac{2(NC_1+C_3)+4Vy_{\max}}{T}+\frac{2NC_2V}{T^2}}.
\end{equation}
Sending $T\rightarrow\infty$ gives
\[
\lim_{T\rightarrow\infty}\frac{1}{T}{\expect{Q_l[T]}}=0,~~\forall l\in\{1,2,\cdots,L\}.
\]
Finally, by Lemma \ref{lemma:queue-bound}, all constraints are satisfied.
\end{proof}

Note that the above proof implies a more refined result that illustrates the convergence time. Fix an 
$\varepsilon>0$, let $V=1/\varepsilon$, then, for all $T\geq 1/\varepsilon$, \eqref{final-dpp-2} implies that 
\[\frac1T\sum_{t=0}^{T-1}\sum_{n=1}^N\expect{y^n[t]}
\leq f_*+\mathcal{O}(\varepsilon).\]
However, \eqref{key-bound} suggests a larger convergence time is required for constraint satisfaction!
For $V=1/\varepsilon$, it can be shown that \eqref{key-bound} implies that 
\[\frac{1}{T}\sum_{t=0}^{T-1}\sum_{n=1}^N\expect{z^n_l[t]}\leq d_l + \mathcal{O}(\varepsilon),\]
whenever $T\geq 1/\varepsilon^3$. The next section shows a tighter $1/\varepsilon^2$ convergence time with a mild Lagrange multiplier assumption.

\section{Convergence Time Analysis}\label{sec-convergence-time}
\subsection{Lagrange Multipliers}
Consider the following optimization problem:
\begin{align}
\min&~~\sum_{n=1}^N\overline{f}^n\label{modi-prob-1}\\
s.t.&~~\sum_{n=1}^N\overline{g}^n_l\leq d_l,~\forall l\in\{1,2,\cdots, L\}, \\
&~~(\overline{f}^n,\overline{\mathbf{g}}^n)\in\mathcal{P}^n,~\forall n \in\{1,2,\cdots,N\}.
\label{modi-prob-3}
\end{align}
Since $\mathcal{P}^n$ is convex, it follows $\mathcal{P}^n$ is convex and $\otimes_{n=1}^N\mathcal{P}^n$ is also convex. Thus, \eqref{modi-prob-1}-\eqref{modi-prob-3} is a convex program. Furthermore, by Lemma \ref{stationary-lemma}, we have \eqref{modi-prob-1}-\eqref{modi-prob-3} is feasible if and only if \eqref{prob-1}-\eqref{prob-2} is feasible, and when assuming feasibility, they have the same optimality $f_*$ as is specified in Lemma \ref{stationary-lemma}.

Since $\mathcal{P}^n$ is convex, one can show (see Proposition 5.1.1 of \cite{Be09}) that there \textit{always} exists a sequence $(\gamma_0, \gamma_1,\cdots,\gamma_L)$ so that $\gamma_i\geq0,~i=0,1,\cdots,L$ and
\[
\sum_{n=1}^N\gamma_0\overline{f}^n+\sum_{l=1}^L\gamma_l\sum_{n=1}^N\overline{g}^n_l
\geq \gamma_0f_*+\sum_{l=1}^L\gamma_ld_l,
~\forall (\overline{f}^n,\overline{\mathbf{g}}^n)\in\mathcal{P}^n,
\]
i.e. there always exists  a hyperplane parametrized by  $(\gamma_0, \gamma_1,\cdots,\gamma_L)$, supported at $(f_*,d_1,\cdots,d_L)$ and containing the set $\left\{\left(\sum_{n=1}^N\overline{f}^n,~\sum_{n=1}^N\overline{\mathbf{g}}^n\right):~(\overline{f}^n,\overline{\mathbf{g}}^n)\in\mathcal{P}^n,~\forall n\in\{1,2,\cdots,N\}\right\}$ on one side. This hyperplane is called ``separating hyperplane''.
The following assumption stems from this property and simply assumes this separating hyperplane to be non-vertical (i.e. $\gamma_0>0$):
\begin{assumption}\label{sep-hype}
There exists non-negative finite constants $\gamma_1,~\gamma_2,~\cdots,~\gamma_L$ such that the following holds,
\begin{align*}
\sum_{n=1}^N\overline{f}^n+\sum_{l=1}^L\gamma_l\sum_{n=1}^N\overline{g}^n_l
\geq f_*+\sum_{l=1}^L\gamma_ld_l,
~\forall (\overline{f}^n,\overline{\mathbf{g}}^n)\in\mathcal{P}^n,
\end{align*}
i.e. there exists a separating hyperplane parametrized by $(1,\gamma_1,\cdots,\gamma_L)$.
\end{assumption} 

\begin{remark}
The parameters $\gamma_1,~\cdots,~\gamma_L$ are called Lagrange multipliers and this assumption is equivalent to the existence of Lagrange multipliers for 
constrained convex program \eqref{modi-prob-1}-\eqref{modi-prob-3}. 
It is known that Lagrange multipliers exist if the Slater's condition holds (\cite{Be09}), which states that there exists a nonempty interior of the feasible region for the convex program. Slater's condition is very common in convex optimization theory and plays an important role in convergence rate analysis, such as the analysis of the interior point algorithm (\cite{BV04}). In the current context, this condition is satisfied, for example, in energy aware server scheduling problems, if the highest possible sum of service rates from all servers is strictly higher than the arrival rate.
\end{remark}

\begin{lemma}\label{bound-lemma-2}
Suppose $\{y^n[t]\}_{t=0}^\infty$, $\{\mathbf{z}^n[t]\}_{t=0}^\infty$ and $\{T^n_k\}_{k=0}^\infty$ are processes resulting from the proposed algorithm.
Under the Assumption \ref{sep-hype}, 
\begin{align*}
\frac1T\sum_{t=0}^{T-1}\left(f_*-\sum_{n=1}^N\expect{y^n[t]}\right)
\leq\frac1T\sum_{t=0}^{T-1}\sum_{l=1}^L\gamma_l\left(\sum_{n=1}^N\expect{z^n_l[t]}-d_l\right)
+\frac{C_4}{T},
\end{align*}
where $C_4=B_1N+B_2N\sum_{l=1}^L\gamma_l$, and $B_1$, $B_2$ are defined in Lemma \ref{bound-lemma-1}.
\end{lemma}

\begin{proof}
First of all, from the statement of Lemma \ref{bound-lemma-1}, for the proposed algorithm, we can define the corresponding processes $(f^n[t],\mathbf{g}^n[t])$ for all $n$ as
\begin{align*}
f^n[t] =& \widehat{f}^n(\alpha^n) = \widehat{y}^n(\alpha^n)/\widehat{T}^n(\alpha^n),
~~\textrm{if}~t\in\mathcal{T}^n_k,\alpha^n_k=\alpha^n\\
\mathbf{g}^n[t] =& \widehat{\mathbf{g}}^n(\alpha^n)= \widehat{\mathbf z}^n(\alpha^n)/\widehat{T}^n(\alpha^n),~~\textrm{if}~t\in\mathcal{T}^n_k,\alpha^n_k=\alpha^n,
\end{align*}
where the last equality follows from the definition of $\widehat{f}^n(\alpha^n)$ and $\widehat{\mathbf{g}}^n(\alpha^n)$ in Definition \ref{PV-def}.
Since
$\left(\widehat{y}^n(\alpha^n),~\widehat{\mathbf z}^n(\alpha^n),~\widehat{T}^n(\alpha^n)\right)
\in\mathcal{S}^n$, by definition of $\mathcal{P}^n$ in Definition \ref{PR-def},
$(f^n[t],\mathbf{g}^n[t])\in\mathcal P^n\subseteq\mathcal{P}^n,~\forall n,~\forall t$.
Since $\mathcal{P}^n$ is a convex set by Lemma \ref{convex-lemma}, it follows
\begin{align*}
\left(\expect{f^n[t]},~\expect{\mathbf{g}^n[t]}\right)\in\mathcal{P}^n,~~\forall t,~\forall n.
\end{align*}
By Assumption \ref{sep-hype}, we have
\begin{align*}
\sum_{n=1}^N\expect{f^n[t]}+\sum_{l=1}^L\gamma_l\sum_{n=1}^N\expect{g^n_l[t]}
\geq f_*+\sum_{l=1}^L\gamma_ld_l,~~\forall t.
\end{align*}
Rearranging terms gives
\begin{align*}
f_*-\sum_{n=1}^N\expect{f^n[t]}\leq \sum_{l=1}^L\gamma_l\left(\sum_{n=1}^N\expect{g^n_l[t]}-d_l\right),~~\forall t.
\end{align*}
Taking the time average from 0 to $T-1$ gives
\begin{align}
\frac1T\sum_{t=0}^{T-1}\left(f_*-\sum_{n=1}^N\expect{f^n[t]}\right)
 \leq \frac1T\sum_{t=0}^{T-1}\sum_{l=1}^L\gamma_l\left(\sum_{n=1}^N\expect{g^n_l[t]}-d_l\right). \label{inter-ave-bound-1}
\end{align}
For the left hand side of \eqref{inter-ave-bound-1}, we have
\begin{align}
l.h.s.&=\frac1T\sum_{t=0}^{T-1}\left(f_*-\sum_{n=1}^N\expect{y^n[t]}\right)
+\frac1T\sum_{t=0}^{T-1}\sum_{n=1}^N\expect{y^n[t]-f^n[t]}\nonumber\\
&\geq\frac1T\sum_{t=0}^{T-1}\left(f_*-\sum_{n=1}^N\expect{y^n[t]}\right)-\frac{B_1N}{T}.\label{inter-ave-bound-2}
\end{align}
where the inequality follows from \eqref{bound-1} in Lemma \ref{bound-lemma-1}.
For the right hand side of \eqref{inter-ave-bound-1}, we have
\begin{align}
r.h.s.&= \frac1T\sum_{t=0}^{T-1}\sum_{l=1}^L\gamma_l\left(\sum_{n=1}^N\expect{z^n_l[t]}-d_l\right)
+\frac1T\sum_{t=0}^{T-1}\sum_{l=1}^L\gamma_l\sum_{n=1}^N\expect{g^n_l[t]-z^n_l[t]}\nonumber\\
&\leq\frac1T\sum_{t=0}^{T-1}\sum_{l=1}^L\gamma_l\left(\sum_{n=1}^N\expect{z^n_l[t]}-d_l\right)+\frac{B_2N\sum_{l=1}^L\gamma_l}{T}, \label{inter-ave-bound-3}
\end{align}
where the inequality follows from the fact that $\gamma_l\geq0,~\forall l$ and \eqref{bound-2} in Lemma \ref{bound-lemma-1}. Substituting \eqref{inter-ave-bound-2} and \eqref{inter-ave-bound-3} into \eqref{inter-ave-bound-1} finishes the proof.
\end{proof}

\subsection{Convergence time theorem}
\begin{theorem}
Fix $\varepsilon\in(0,1)$ and define $V=1/\varepsilon$. If the problem \eqref{prob-1}-\eqref{prob-2} is feasible and the Assumption \ref{sep-hype} holds, then, for all $T\geq1/\varepsilon^2$, 
\begin{align}
&\frac1T\sum_{t=0}^{T-1}\sum_{n=1}^N\expect{y^n[t]}\leq f_*+\mathcal{O}(\varepsilon),\label{ctime-1}\\
&\frac1T\sum_{t=0}^{T-1}\sum_{n=1}^N\expect{z^n_l[t]}\leq d_l+\mathcal{O}(\varepsilon),
l\in\{1,2,\cdots,L\}.\label{ctime-2}
\end{align}
Thus, the algorithm provides $\mathcal{O}(\varepsilon)$ approximation with the convergence time $\mathcal{O}(1/\varepsilon^2)$.
\end{theorem}
\begin{proof}
First of all, by queue updating rule \eqref{queue-update}, 
\begin{equation}\label{queue-bound}
\sum_{t=0}^{T-1}\left(\sum_{n=1}^N\expect{z^n_l[t]}-d_l\right)\leq\expect{Q_l[T]}.
\end{equation}
By Lemma \ref{bound-lemma-2}, we have
\begin{align}
\frac1T\sum_{t=0}^{T-1}\left(f_*-\sum_{n=1}^N\expect{y^n[t]}\right)
\leq&\frac1T\sum_{t=0}^{T-1}\sum_{l=1}^L\gamma_l\left(\sum_{n=1}^N\expect{z^n_l[t]}-d_l\right)
+\frac{C_4}{T},\nonumber\\
\leq&\sum_{l=1}^L\frac{\gamma_l}{T}\expect{Q_l[T]}+\frac{C_4}{T}.\label{inter-ctime-1}
\end{align}
Combining this with \eqref{final-dpp} gives
\begin{align}
\frac{1}{2T}\expect{\|\mathbf{Q}[T]\|^2}
&\leq NC_1+C_3+\frac{V}{T}\sum_{t=0}^{T-1}\left(f_*-\sum_{n=1}^N\expect{y^n[t]}\right)+\frac{NC_2V}{T}\nonumber\\
&\leq NC_1+C_3 +\frac{(NC_2+C_4)V}{T}+V\sum_{l=1}^L\frac{\gamma_l}{T}\expect{Q_l[T]}
\nonumber\\
&\leq NC_1+C_3 +\frac{(NC_2+C_4)V}{T}+\frac{V}{T}\|\gamma\|\cdot\|\expect{\mathbf{Q}[T]}\|,
\label{inter-ctime-2}
\end{align}
where $\gamma:=(\gamma_1,~\cdots,~\gamma_L)$,
the second inequality follows from \eqref{inter-ctime-1} and the final inequality follows from Cauchy-Schwarz. Then, by Jensen's inequality, we have
\[\|\expect{\mathbf{Q}[T]}\|^2\leq\expect{\|\mathbf{Q}[T]\|^2}.\]
Thus, it follows by \eqref{inter-ctime-2} that
\begin{align*}
\|\expect{\mathbf{Q}[T]}\|^2- 2V\|\gamma\|\cdot\|\expect{\mathbf{Q}[T]}\| - 2(NC_1+C_3)T
-2(NC_2+C_4)V\leq 0.
\end{align*}
The left hand side is a quadratic form on $\|\expect{\mathbf{Q}[T]}\|$, and the inequality implies that 
$\|\expect{\mathbf{Q}[T]}\|$ is deterministically upper bounded by the largest root of the equation 
$x^2-bx-c=0$ with $b=2V\|\gamma\|$ and $c=2(NC_1+C_3)T+2(NC_2+C_4)V$. Thus,
\begin{align*}
\|\expect{\mathbf{Q}[T]}\|\leq&\frac{b+\sqrt{b^2+4c}}{2}\\
=& V\|\gamma\|+\sqrt{V^2\|\gamma\|^2+2(NC_1+C_3)T+2(NC_2+C_4)V}\\
\leq& 2V\|\gamma\|+ \sqrt{2(NC_1+C_3)T} + \sqrt{2(NC_2+C_4)V}.
\end{align*}
Thus, for any $l\in\{1,2,\cdots,L\}$,
\begin{align*}
\frac1T\expect{Q_l[T]}\leq \frac{2V\|\gamma\|}{T} 
 + \sqrt{\frac{2(NC_1+C_3)}{T}} + \frac{\sqrt{2(NC_2+C_4)V}}{T}.
\end{align*}
By \eqref{queue-bound} again,
\begin{align*}
\frac1T\sum_{t=0}^{T-1}\sum_{n=1}^N\expect{z^n_l[t]}\leq d_l+\frac{2V\|\gamma\|}{T} 
 + \sqrt{\frac{2(NC_1+C_3)}{T}} + \frac{\sqrt{2(NC_2+C_4)V}}{T}.
\end{align*}
Substituting $V=1/\varepsilon$ and $T\geq1/\varepsilon^2$ into the above inequality gives 
$\forall l\in\{1,2,\cdots,L\}$,
\begin{align*}
\frac1T\sum_{t=0}^{T-1}\sum_{n=1}^N\expect{z^n_l[t]}\leq&
d_l + \left(2\|\gamma\|+\sqrt{2(NC_1+C_3)}\right)\varepsilon + \sqrt{2(NC_2+C4)}\varepsilon^{3/2}\\
=&d_l+\mathcal{O}(\varepsilon).
\end{align*}
Finally, substituting $V=1/\varepsilon$ and $T\geq1/\varepsilon^2$ into \eqref{final-dpp-2} gives
\[\frac1T\sum_{t=0}^{T-1}\sum_{n=1}^N\expect{y^n[t]}\leq f_*+\mathcal{O}(\varepsilon),\]
finishing the proof.
\end{proof}

\section{Simulation Study in Energy-aware Scheduling}\label{section-application}
 Here, we apply the algorithm introduced in Section \ref{section:algorithm} to deal with the energy-aware scheduling problem described in Section \ref{sec:application}. To be specific, we consider a scenario with 5 homogeneous servers and 3 different classes of jobs, i.e. $N=5$ and $L=3$.
We assume that each server can only choose one class of jobs to serve during each frame. So the mode set $\mathcal{M}^n$ contains three actions $\{1,2,3\}$ and the action $i$ stands for serving the $i$-th class of jobs and we count the number of serviced jobs at the end of each service duration. The action $m^n_k$ determines the following quantities:
 \begin{itemize}
\item The uniformly distributed total number of class $l$ jobs that can be served with expectation
$\expect{\left.\sum_{t\in\mathcal{T}^n_k}\mu^n_l[t]\right|~m^n_k}:=\widehat{\mu}^n_l(m^n_k)$.
\item The geometrically distributed service duration $H^n_k$ slots with expectation $\expect{\left.H^n_k\right|~m^n_k}:=\widehat{H}^n(m^n_k)$.
\item The energy consumption $\widehat{e}^n(m^n_k)$ for serving all these jobs.
\item The geometrically distributed idle/setup time $I^n_k$ slots with constant energy consumption $p^n$ per slot and zero job service. The expectation $\expect{\left.I^n_k\right|~m^n_k}:=\widehat{I}^n(m^n_k)$.
\end{itemize}
The idle/setup cost is $p^n=3$ units per slot and the rest of the parameters are listed in Table 1.

Following the algorithm description in Section \ref{section:algorithm}, the proposed algorithm has the queue updating rule 
\[Q_l[t+1]=\max\left\{Q_l[t]+\lambda_l[t]-\sum_{n=1}^N\mu^n_l[t],~0\right\},\]
and each system minimizes
\eqref{DPP-ratio} each frame, which can be written as 
\begin{align*}
\min_{m^n_k\in\mathcal{M}^n}\frac{V\left(\widehat{e}^n_l(m^n_k)+p^n\widehat{I}^n(m^n_k)\right)
-\dotp{\mathbf{Q}[t^n_k]}{\widehat{\mu}^n(m^n_k)}}{\widehat{H}^n(m^n_k)+\widehat{I}^n(m^n_k)}.
\end{align*}

\begin{table}
\begin{center}
\caption{Problem parameters}
\begin{tabular}{c|c|c|c|c|c}
  \hline
   & $\lambda_i$ & $\widehat{H}^n(i)$ & $\widehat{\mu}^n(i)$ & $\widehat{e}^n(i)$ & $\widehat{I}^n(i)$\\
   \hline
Class 1   & 2 & 5.5 & 15 (Uniform $[9,21]\cap\mathbb{N}$) & 16 & 2.5 \\
  \hline
 Class 2  & 3 & 4.6 & 21 (Uniform $[15,27]\cap\mathbb{N}$) & 20 & 4.3\\
 \hline
 Class 3  & 4 & 3.8 & 17 (Uniform $[11,23]\cap\mathbb{N}$) & 13 & 3.7\\
 \hline
\end{tabular}
\end{center}
\end{table}

Each plot for the proposed algorithm is the result of running 1 million slots and taking the time average as the performance of the proposed algorithm. The benchmark is the optimal stationary performance obtained by performing a change of variable and solving a linear program, knowing the arrival rates (see also \cite{Neely12} for details). 

Fig. \ref{fig:Stupendous2} shows as the trade-off parameter $V$ gets larger, the time average energy consumptions under the proposed algorithm approaches the optimal energy consumption. Fig. \ref{fig:Stupendous3} shows as $V$ gets large, the time average number of services also approaches the optimal service rate for each class of jobs. In Fig. \ref{fig:Stupendous4}, we plot the time average queue backlog for each class of jobs verses $V$ parameter. We see that the queue backlog for the first class is always low whereas the rest queue backlogs scale up linearly with $V$. This is because the service rate for the first class is always strictly larger than the arrival rate whereas for the rest classes, as $V$ gets larger, the service rates approach the arrival rates. This plot, together with Fig. \ref{fig:Stupendous2}, also demonstrate that $V$ 
is indeed a trade-off parameter which trades queue backlog for near optimality. 
\begin{figure}[htbp]
   \centering
   \includegraphics[height=3in]{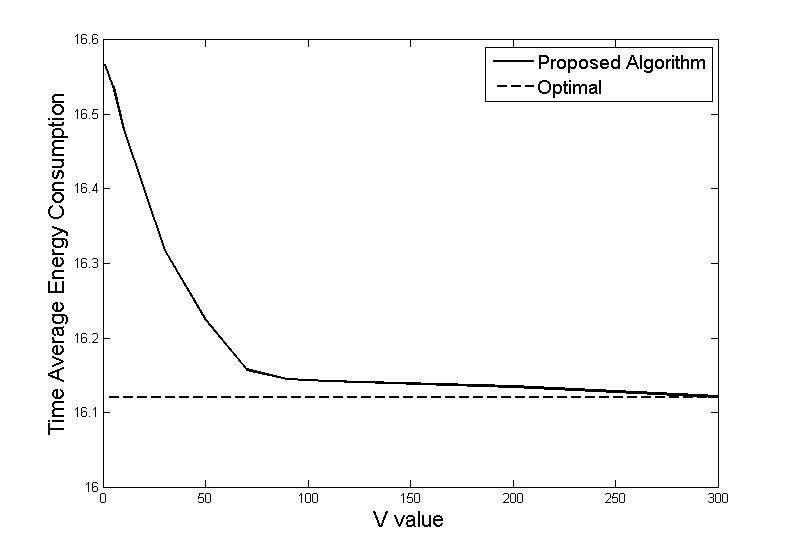} 
   \caption{Time average energy consumption verses $V$ parameter over 1 millon slots.}
   \label{fig:Stupendous2}
\end{figure}

\begin{figure}[htbp]
   \centering
   \includegraphics[height=3in]{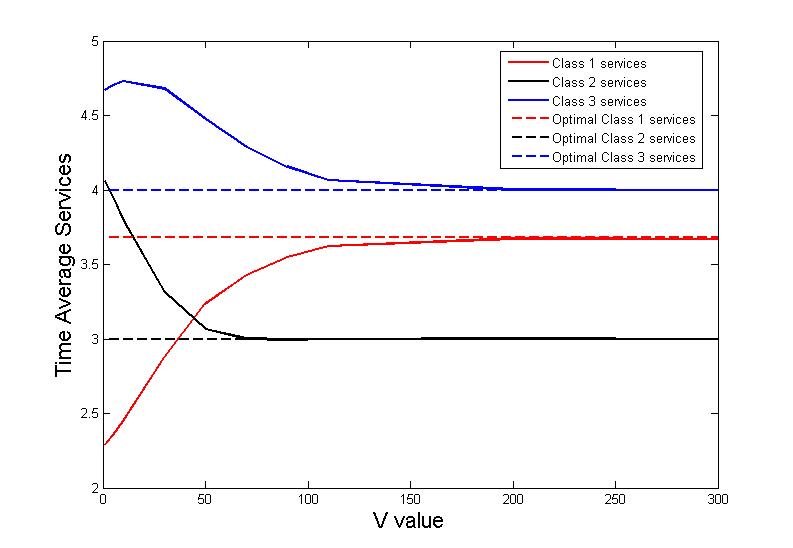} 
   \caption{Time average services verses $V$ parameter over 1 millon slots.}
   \label{fig:Stupendous3}
\end{figure}

\begin{figure}[htbp]
   \centering
   \includegraphics[height=3in]{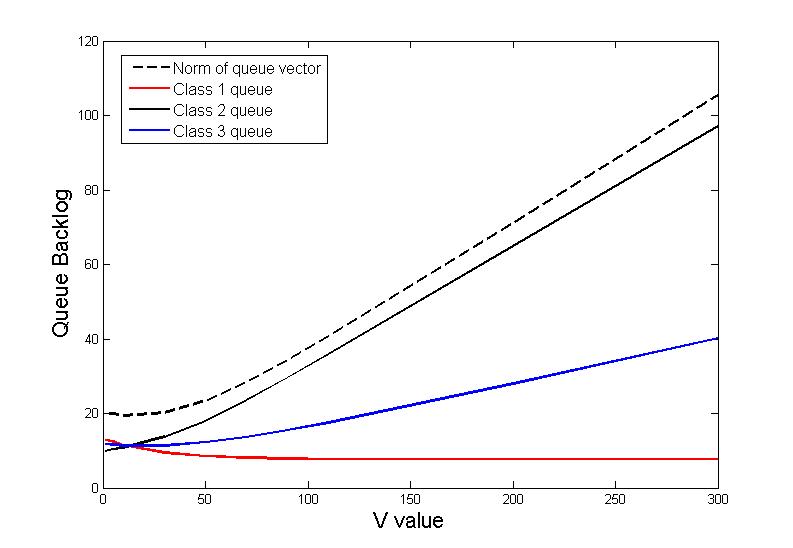} 
   \caption{Time average queue size verses $V$ parameter over 1 million slots.}
   \label{fig:Stupendous4}
\end{figure}

\appendix

\section{Additional lemmas and proofs.}\label{appendix-proof}
\begin{proof}[Proof of Lemma \ref{lemma:queue-bound}]
Fix $l\in\{1,2,\cdots,L\}$. For any fixed $T$, $Q_l[T]=\sum_{t=0}^{T-1}(Q_l[t+1]-Q_l[t])$. For each summand, by queue updating rule \eqref{queue-update},
\begin{align*}
Q_l[t+1]-Q_l[t]=&\max\left\{Q_l[t]+\sum_{n=1}^Nz^n_l[t]-d_l[t],~0\right\}-Q_l[t]\\
\geq&Q_l[t]+\sum_{n=1}^Nz^n_l[t]-d_l[t]-Q_l[t]=\sum_{n=1}^Nz^n_l[t]-d_l[t].
\end{align*}
Thus, by the assumption $Q_l[0]=0$,
$$Q_l[T]\geq\sum_{t=0}^{T-1}\left(\sum_{n=1}^Nz^n_l[t]-d_l[t]\right).$$
Taking expectations of both sides with $\expect{d_l[t]}=d_l,~\forall l$, gives
$$\expect{Q_l[T]}\geq\sum_{t=0}^{T-1}\left(\sum_{n=1}^N\expect{z^n_l[t]}-d_l\right).$$
Dividing both sides by $T$ and passing to the limit gives
\[\limsup_{T\rightarrow\infty}\frac1T\sum_{t=0}^{T-1}\left(\sum_{n=1}^N\expect{z^n_l[t]}-d_l\right)
\leq\lim_{T\rightarrow\infty}\frac{1}{T}{\expect{Q_l[T]}}=0,\]
finishing the proof.
\end{proof}

\begin{proof}[Proof of Lemma \ref{bound-lemma-1}]
We prove bound \eqref{bound-1} (\eqref{bound-2} is proved similarly). By definition of $\widehat{f}^n(\alpha^n)$ in Definition \ref{PV-def}, we have for any $\alpha^n\in\mathcal{A}^n$,
\[\widehat{f}^n(\alpha^n)=\frac{\expect{\left.\sum_{t\in\mathcal{T}^n_k}y^n[t]\right|~\alpha^n_k=\alpha^n}}{\expect{T^n_k|~\alpha^n_k=\alpha^n}},\]
thus,
\[\expect{\left.\sum_{t\in\mathcal{T}^n_k}\left(\widehat{f}^n(\alpha^n_k)-y^n[t]\right)\right|~\alpha^n_k=\alpha^n}=0.\]
By the renewal property of the system, given $\alpha^n_k=\alpha^n$, $T^n_k$ and $\sum_{t\in\mathcal{T}^n_k}y^n[t]$ are independent of the past information before $t^n_k$. Thus, the same equality holds if conditioning also on $\mathcal F_k^n $, i.e.
\[\expect{\left.\sum_{t\in\mathcal{T}^n_k}\left(\widehat{f}^n(\alpha^n_k)-y^n[t]\right)\right|~\alpha^n_k=\alpha^n,~\mathcal F_k^n }=0.\]
Hence,
\[\expect{\left.\sum_{t\in\mathcal{T}^n_k}\left(\widehat{f}^n(\alpha^n_k)-y^n[t]\right)\right|~\mathcal F_k^n }=0.\]

By the definition of $f^n[t]$, this further implies that
\[\expect{\left.\sum_{t\in\mathcal{T}^n_k}\left(f^n[t]-y^n[t]\right)\right|~\mathcal F_k^n }=0.\]
Since $|y^n[t]|\leq y_{\max}$ and $\expect{T^n_k}\leq\sqrt{B}$, it follows $\expect{\left|\sum_{t\in\mathcal{T}^n_k}\left(f^n[t]-y^n[t]\right)\right|}<\infty$ and 
the process $\{F^n_K\}_{K=0}^\infty$ defined as
\[F^n_K=\sum_{k=0}^{K-1}\sum_{t\in\mathcal{T}^n_k}\left(f^n[t]-y^n[t]\right),~K\geq1,\]
$F^n_0=0$ is a \textit{martingale}. 

Consider any fixed $T\in\mathbb{N}$ and define $S^n[T]$ as the number of renewals up to $T$. Lemma \ref{valid-stopping-time} shows $S^n[T]$ is a valid stopping time with respect to the filtration $\{\mathcal F_k^n \}_{k=0}^\infty$.
Furthermore, $\{F^n_{K\wedge S^n[T]}\}_{K=0}^\infty$ is a supermartingale by Theorem \ref{stopping-time}, where $a\wedge b:=\min\{a,b\}$.

For this fixed $T$, we have
\begin{align*}
\expect{\sum_{t=0}^{T-1}\left(f^n[t]-y^n[t]\right)}
=&\expect{\sum_{t=0}^{t^n_{S^n[T]}-1}\left(f^n[t]-y^n[t]\right)}
-\expect{\sum_{t=T}^{t^n_{S^n[T]}-1}\left(f^n[t]-y^n[t]\right)}\\
=&\expect{F^n_{S^n[T]}}-\expect{\sum_{t=T}^{t^n_{S^n[T]}-1}\left(f^n[t]-y^n[t]\right)}.
\end{align*}
Since the number of renewals is always bounded by the number of slots at any time, i.e. $S^n[T]\leq T+1$, it follows
\[\expect{F^n_{S^n[T]}}=\expect{F^n_{(T+1)\wedge S^n[T]}}\leq 0.\]
On the other hand, 
\begin{align*}
\left|\expect{\sum_{t=T}^{t^n_{S^n[T]}-1}\left(f^n[t]-y^n[t]\right)}\right|
\leq\expect{t^n_{S^n[T]}-T}\cdot2y_{\max}\leq2y_{\max}\sqrt{B}.
\end{align*}
where the last inequality follows from Assumption \ref{bounded-assumption} for the residual life time. Thus, 
\[\expect{\sum_{t=0}^{T-1}\left(f^n[t]-y^n[t]\right)}\leq2y_{\max}\sqrt{B}.\]
Dividing both sides by $T$ finishes the proof.
\end{proof}

\begin{proof}[Proof of Lemma \ref{lemma:filtration}]
Recall that $t^n_k$ is the time slot where the $k$-th renewal occurs ($k=0,1,2,\cdots$), then, it follows from the definition of stopping time (\cite{Durrett}) that $\{t^n_k\}_{k=0}^\infty$ is a sequence of stopping times with respect to $\{\mathcal{F}[t]\}_{t=0}^{\infty}$ satisfying $t^n_k<t^n_{k+1},~\forall k$. Thus, by definition of 
$\mathcal F_k^n $, for any set $A\in\mathcal F_k^n $,
\[A\cap\{t^n_{k+1}\leq t\}=A\cap\{t^n_k\leq t\}\cap\{t^n_{k+1}\leq t\}\in\mathcal{F}[t].\]
Thus, $A\in\mathcal F_{k+1}^n $, which implies $\mathcal F_k^n \subseteq\mathcal F_{k+1}^n ,~\forall k$,
and $\{\mathcal F_k^n \}_{k=0}^\infty$ is indeed a filtration. This finishes the first part of the proof.

Next,
we would like to show that $G_{t^n_k}(Z^n_0,\cdots,Z^n[t^n_k-1])$ is measurable with respect to $\mathcal F_k^n ,~\forall k\geq1$,
i.e. $\left\{G_{t^n_k}(Z^n_0,\cdots,Z^n[t^n_k-1])\in B\right\}\in\mathcal F_k^n $, for any Borel set $B\subseteq\mathbb{R}$. By definition of 
$\mathcal F_k^n $, this is equivalent to showing $\{G_{t^n_k}(Z^n_0,\cdots,Z^n[t^n_k-1])\in B\} \cap \{t^n_k\leq s\}\in\mathcal{F}[s]$ for any slot $s\geq0$. For $s=0$, this is obvious because $ \{t^n_k\leq 0\}=\emptyset,~\forall k\geq1$. Consider any $s\geq1$,
\begin{align*}
&\left\{G_{t^n_k}(Z^n_0,\cdots,Z^n[t^n_k-1])\in B\right\} \cap \{t^n_k\leq s\}\\
&= \bigcup_{i=1}^{s}\left(\left\{G_{i}(Z^n_0,\cdots,Z^n[i-1])  \in B\right\}\bigcap\{t^n_k = i\}\right)\\
&= \bigcup_{i=1}^{s}\left(\left\{(Z^n_0,\cdots,Z^n[i-1])  \in G^{-1}_i(B)\right\}\bigcap\{t^n_k = i\}\right)
 \in\mathcal{F}[s], ~\forall k\geq1,
\end{align*}
where the last step follows from the assumption that
the random variable $Z^n[t-1]$ is measurable with respect to $\mathcal{F}[t]$ for any $t>0$ and $t^n_k$ is a stopping time with respect to $\{\mathcal{F}[t]\}_{t=0}^{\infty}$ for all $k\geq1$. This gives the second part of the claim.
\end{proof}

\begin{proof}[Proof of Lemma \ref{valid-stopping-time}]
We aim to prove $\{S^n[T]= k\}\in\mathcal F_k^n ,~\forall k\in\mathbb{N}$.
First of all, recall that the index of the renewal starts from $k=0$ and $t^n_0=0$, thus, 
for any $k\in\mathbb{N}$, 
$\{S^n[T]=k\} = \{t^n_k> T\}\cap\{t^n_{k-1}\leq T\}$,
and any $t\in\mathbb{N}$,
\begin{align}
\{S^n[T]=k\}\cap\{t^n_k\leq t\}
=&\{t^n_k> T\}\cap\{t^n_{k-1}\leq T\}\cap\{t^n_k\leq t\}. \label{the-set}
\end{align}

Consider two cases as follows:
\begin{enumerate}
\item $t\leq T$. In this case, the set \eqref{the-set} is empty and obviously belongs to $\mathcal{F}[t]$.
\item $t>T$. In this case, we have $\{t^n_k> T\}\cap\{t^n_k\leq t\}=\{T<t^n_k\leq t\}\in\mathcal{F}[t]$ as well as $\{t^n_{k-1}\leq T\}\in\mathcal{F}[T]\subseteq\mathcal{F}[t]$. Thus, the set \eqref{the-set} belongs to $\mathcal{F}[t]$.
\end{enumerate}
Overall, we have $\{S^n[T]=k\}\cap\{t^n_k\leq t\}\in \mathcal{F}[t],~\forall t\in\mathbb{N}$. Thus,
$\{S^n[T]=k\}\in\mathcal F_k^n $ and $S^n[T]$ is indeed a valid stopping time with respect to the filtration $\{\mathcal F_k^n \}_{k=0}^\infty$.
\end{proof}

The rest of the section is devoted to the proof of Lemma \ref{stationary-lemma}.

\begin{proof}[Proof of Lemma \ref{stationary-lemma}]
To prove the first part of the claim, we define the following notation:
\[\bigoplus_{n=1}^N\mathcal{P}^n:=\left\{\sum_{n=1}^N\mathbf{p}_n,~\mathbf{p}_n\in\mathcal{P}^n,~\forall n
\right\}\]
is the Minkowski sum of sets $\mathcal{P}_n,~n\in\{1,2,\cdots,N\}$, and for any sequence $\{\mathbf{x}[t]\}_{t=0}^\infty$ taking values in 
$\mathbb{R}^d$, define
$$\limsup_{T\rightarrow\infty}\mathbf{x}[T]:=
\left(\limsup_{T\rightarrow\infty}x_1[T],~\cdots,\limsup_{T\rightarrow\infty}x_d[T]\right)$$ 
is a vector of $\limsup$s. By definition, any vector in $\oplus_{n=1}^N\mathcal{P}^n$ can be constructed from $\otimes_{n=1}^N\mathcal{P}^n$, thus, it is enough to show that there exists a vector $\mathbf{r}^*\in\oplus_{n=1}^N\mathcal{P}^n$ such that $r_0^*=f^*$ and the rest of the entries $r^*_l\leq d_l,~l=1,2,\cdots,L$.

By the feasibility assumption for \eqref{prob-1}-\eqref{prob-2}, we can
consider \textit{any algorithm that achieves the optimality} of \eqref{prob-1}-\eqref{prob-2} and the corresponding process $\{(f^n[t],\mathbf{g}^n[t])\}_{t=0}^\infty$ defined in Lemma \ref{bound-lemma-1} for any system $n$. Notice that $(f^n[t],\mathbf{g}^n[t])\in\mathcal{P}^n,~\forall n,~\forall t$. This follows from the definition of $\widehat{f}^n(\alpha^n)$ and $\widehat{\mathbf{g}}^n(\alpha^n)$ in Definition \ref{PV-def} that
\begin{align*}
f^n[t] =& \widehat{f}^n(\alpha^n) = \widehat{y}^n(\alpha^n)/\widehat{T}^n(\alpha^n),
~~\textrm{if}~t\in\mathcal{T}^n_k,\alpha^n_k=\alpha^n\\
\mathbf{g}^n[t] =& \widehat{\mathbf{g}}^n(\alpha^n)= \widehat{\mathbf z}^n(\alpha^n)/\widehat{T}^n(\alpha^n),~~\textrm{if}~t\in\mathcal{T}^n_k,\alpha^n_k=\alpha^n,
\end{align*}
and 
$\left(\widehat{y}^n(\alpha^n),~\widehat{\mathbf z}^n(\alpha^n),~\widehat{T}^n(\alpha^n)\right)
\in\mathcal{S}^n$. By definition of $\mathcal{P}^n$ in Definition \ref{PR-def},
$(f^n[t],\mathbf{g}^n[t])\in\mathcal P^n,~\forall n,~\forall t$.

Since $\mathcal{P}^n$ is convex by Lemma \ref{convex-lemma}, it follows that 
$\left(\expect{f^n[t]},\expect{\mathbf{g}^n[t]}\right)\in\mathcal{P}^n,~\forall n,~\forall t$. Hence,
\[\left(\frac1T\sum_{t=1}^{T-1}\expect{f^n[t]},~\frac1T\sum_{t=1}^{T-1}\expect{\mathbf{g}^n[t]}\right)\in\mathcal{P}^n,~\forall T, \forall n.\]
This further implies that
\[
\mathbf{r}(T):=\left(\frac1T\sum_{t=1}^{T-1}\sum_{n=1}^N\expect{f^n[t]},~\frac1T\sum_{t=1}^{T-1}\sum_{n=1}^N\expect{\mathbf{g}^n[t]}\right)\in\bigoplus_{n=1}^N\mathcal{P}^n.
\]
By Lemma \ref{convex-lemma}, $\mathcal P^n$ is compact in $\mathbb{R}^{L+1}$. Thus,
$\oplus_{n=1}^N\mathcal{P}^n$ is also compact.
This implies that the sequence $\{\mathbf{r}(T)\}_{T=1}^\infty$ has at least one limit point, and any such limit point is contained in $\oplus_{n=1}^N\mathcal{P}^n$. 

We consider a specific limit point of $\{\mathbf{r}(T)\}_{T=1}^\infty$ denoted as $\mathbf{r}^*\in\oplus_{n=1}^N\mathcal{P}^n$, with the first entry denoted as $r_0^*$ satisfying 
$$r_0^* = \limsup_{T\rightarrow\infty}\frac1T\sum_{t=0}^{T-1}\sum_{n=1}^N\expect{f^n[t]}.$$
Then, we have the rest of the entries of $\mathbf{r}^*$ must satisfy
\[r_l^*\leq\limsup_{T\rightarrow\infty}\frac1T\sum_{t=0}^{T-1}\sum_{n=1}^N\expect{\mathbf{g}^n[t]},
~\forall l\in\{1,2,\cdots,L\}.\]
Now, 
by Lemma \ref{bound-lemma-1}, we can connect the $\limsup$ with respect to $f^n[t]$ and 
$\mathbf{g}^n[t]$ to that of $y^n[t]$ and $\mathbf{z}^n[t]$ as follows:
\begin{align*}
&\limsup_{T\rightarrow\infty}\frac1T\sum_{t=0}^{T-1}\sum_{n=1}^N\expect{y^n[t]}\\
=&\limsup_{T\rightarrow\infty}\frac1T\sum_{t=0}^{T-1}\sum_{n=1}^N\left(\expect{y^n[t]-f^n[t]}+\expect{f^n[t]}\right)\\
=&\lim_{T\rightarrow\infty}\frac1T\sum_{t=0}^{T-1}\sum_{n=1}^N\expect{y^n[t]-f^n[t]}
+\limsup_{T\rightarrow\infty}\frac1T\sum_{t=0}^{T-1}\sum_{n=1}^N\expect{f^n[t]}\\
=&\limsup_{T\rightarrow\infty}\frac1T\sum_{t=0}^{T-1}\sum_{n=1}^N\expect{f^n[t]}.
\end{align*}
Similarly, we can show that
\[
\limsup_{T\rightarrow\infty}\frac1T\sum_{t=0}^{T-1}\sum_{n=1}^N\expect{\mathbf{z}^n[t]}
=\limsup_{T\rightarrow\infty}\frac1T\sum_{t=0}^{T-1}\sum_{n=1}^N\expect{\mathbf{g}^n[t]}.\]
Thus, by our preceeding assumption that the algorithm under consideration achieves the optimality of \eqref{prob-1}-\eqref{prob-2}, we have 
\begin{align*}
&r_0^*=\limsup_{T\rightarrow\infty}\frac1T\sum_{t=0}^{T-1}\sum_{n=1}^N\expect{y^n[t]}=f^*\\
&r_l^*\leq\limsup_{T\rightarrow\infty}\frac1T\sum_{t=0}^{T-1}\sum_{n=1}^N\expect{z_l^n[t]}
\leq d_l,~\forall i\in\{1,2,\cdots,L\}.
\end{align*}
Overall, we have shown that $\mathbf{r}^*\in\oplus_{n=1}^N\mathcal{P}^n$ achieves the optimality of \eqref{prob-1}-\eqref{prob-2}, and the first part of the lemma is proved.

To prove the second part of the lemma, we show that any point in $\otimes_{n=1}^N\mathcal{P}^n$ is achievable by the corresponding time averages of some algorithm. Specifically, consider the following class of \textit{randomized stationary algorithms}: For each system $n$, at the beginning of $k$-th frame, the controller independently chooses an action $\alpha^n_k$ from the set $\mathcal{A}^n$ with a fixed probability distribution. 

Thus, the actions $\{\alpha^n_k\}_{k=0}^{\infty}$ result from any randomized stationary algorithm is i.i.d.. By the renewal property of each system, we have 
$$\left\{\left(\sum_{t\in\mathcal{T}^n_k}y^n[t],~\sum_{t\in\mathcal{T}^n_k}\mathbf{z}^n[t],~T^n_k\right)\right\}_{k=0}^\infty,$$
is also an i.i.d. process for each system $n$.

Next, we would like to show that any point in $\mathcal{S}^n$ can be achieved by the corresponding expectations of some randomized stationary algorithm.
Recall that $\mathcal{S}^n$ defined in Definition \ref{PR-def} is the convex hull of 
$$\mathcal{G}^n:=\left\{\left(\widehat{y}^n(\alpha^n),~\widehat{\mathbf z}^n(\alpha^n),~\widehat{T}^n(\alpha^n)\right),~\alpha^n\in\mathcal{A}^n\right\} \subseteq\mathbb{R}^{L+2},$$ 
By definition of convex hull, 
 any point $(y,\mathbf{z},T)\in\mathcal{S}^n$, can be written as a convex combination of a finite number of points from the set $\mathcal{G}^n$. Let $\left\{\left(\widehat{y}^n(\alpha^n_i),~\widehat{\mathbf z}^n(\alpha^n_i),~\widehat{T}^n(\alpha^n_i)\right)\right\}_{i=1}^m$ be these points,
then, we have there exists a finite sequence $\{p_i\}_{i=1}^m$, such that
\begin{align*}
&(y,\mathbf{z},T) = \sum_{i=1}^m p_i\cdot\left(\widehat{y}^n(\alpha^n_i),~\widehat{\mathbf z}^n(\alpha^n_i),~\widehat{T}^n(\alpha^n_i)\right),\\
&p_i\geq0,~\sum_{i=1}^mp_i=1.
\end{align*}
We can then use $\{p_i\}_{i=1}^m$ to construct the following randomized stationary algorithm: At the start of each frame $k$, the controller independently chooses action $\alpha_i\in\mathcal{A}^n$ with probability $p_i$ defined above for $i=1,2,\cdots,m$. Then, the one-shot expectation of this particular randomized stationary algorithm on system $n$ satisfies
\[
\left(\expect{\sum_{t\in\mathcal{T}^n_k}y^n[t]},~\expect{\sum_{t\in\mathcal{T}^n_k}\mathbf{z}^n[t]},~\expect{T^n_k}\right)=
 \sum_{i=1}^m p_i\cdot\left(\widehat{y}^n(\alpha^n_i),~\widehat{\mathbf z}^n(\alpha^n_i),~\widehat{T}^n(\alpha^n_i)\right)=(y,\mathbf{z},T),
\]
which implies any point in $\mathcal{S}^n$ can be achieved by the corresponding expectations of a randomized stationary algorithm.

Next, by definition of $\mathcal P^n$ in Definition \ref{PR-def}, any $(\overline{f}^n,\overline{\mathbf{g}}^n)\in\mathcal P^n$ can be written as $(\overline{f}^n,\overline{\mathbf{g}}^n)=(y/T,\mathbf{z}/T)$, where $(y,\mathbf{z},T)\in\mathcal{S}^n$. Thus, it is achievable by the ratio of one-shot expectations from a randomized stationary algorithm, i.e.
\[
\frac{\expect{\sum_{t\in\mathcal{T}^n_k}y^n[t]}}{\expect{T^n_k}}=\frac{y}{T}=\overline{f}^n,~~
\frac{\expect{\sum_{t\in\mathcal{T}^n_k}\mathbf{z}^n[t]}}{\expect{T^n_k}}=\frac{\mathbf{z}}{T}
=\overline{\mathbf{g}}^n.
\]
Now we claim that for $y^n[t]$, $\mathbf{z}^n[t]$ and $T^n_k$ result from the randomized stationary algorithm,
\begin{align}
&\lim_{T\rightarrow\infty}\frac1T\sum_{t=0}^{T-1}\expect{y^n[t]}=\frac{\expect{\sum_{t\in\mathcal{T}^n_k}y^n[t]}}{\expect{T^n_k}},\label{lln-1}\\
&\lim_{T\rightarrow\infty}\frac1T\sum_{t=0}^{T-1}\expect{\mathbf{z}^n[t]}=\frac{\expect{\sum_{t\in\mathcal{T}^n_k}\mathbf{z}^n[t]}}{\expect{T^n_k}}.\label{lln-2}
\end{align}
We prove \eqref{lln-1} and \eqref{lln-2} is shown in a similar way. Consider any fixed $T$, and let $S^n[T]$ be the number of renewals up to (and including) time $T$. Then, from Lemma \ref{valid-stopping-time} in Section \ref{section:limiting}, $S^n[T]$ is a valid stopping time with respect to the filtration 
$\{\mathcal F_k^n \}_{k=0}^\infty$. We write
\begin{equation}\label{split-stop}
\frac1T\sum_{t=0}^{T-1}\expect{y^n[t]}=\frac{1}{T}
\expect{\sum_{k=0}^{S^n[T]}\sum_{t\in\mathcal{T}^n_k}y^n[t]}-\frac1T\expect{\sum_{t=T}^{t^n_{S^n[T]}-1}y^n[t]}.
\end{equation}
For the first part on the right hand side of \eqref{split-stop}, since $\left\{\sum_{t\in\mathcal{T}^n_k}y^n[t]\right\}_{k=0}^{\infty}$ is an i.i.d. process, by Wald's equality (Theorem 4.1.5 of \cite{Durrett}),
\[
\frac{1}{T}
\expect{\sum_{k=0}^{S^n[T]}\sum_{t\in\mathcal{T}^n_k}y^n[t]}=\expect{\sum_{t\in\mathcal{T}^n_k}y^n[t]}
\cdot\frac{\expect{S^n[T]}}{T}.
\] 
By renewal reward theorem (Theorem 4.4.2 of \cite{Durrett}),
\[
\lim_{T\rightarrow\infty}\frac{\expect{S^n[T]}}{T}=\frac{1}{\expect{T^n_k}}.
\]
Thus,
\[
\lim_{T\rightarrow\infty}\frac{1}{T}
\expect{\sum_{k=0}^{S^n[T]}\sum_{t\in\mathcal{T}^n_k}y^n[t]}
=\frac{\expect{\sum_{t\in\mathcal{T}^n_k}y^n[t]}}{\expect{T^n_k}}.
\]
For the second part on the right hand side of \eqref{split-stop}, by Assumption \ref{bounded-assumption},
\[
\left|\expect{\sum_{t=T}^{t^n_{S^n[T]}-1}y^n[t]}\right|\leq y_{\max}\cdot\expect{t^n_{S^n[T]}-T}
\leq\sqrt{B}y_{\max},
\]
which implies $\lim_{T\rightarrow\infty}\frac1T\expect{\sum_{t=T}^{t^n_{S^n[T]}-1}y^n[t]}=0$. Overall, we have \eqref{lln-1} holds.

To this point, we have shown that for any $(\overline{f}^n,\overline{\mathbf{g}}^n)\in\mathcal P^n$, $n\in\{1,2,\cdots,N\}$, there exists a randomized stationary algorithm so that 
\begin{align*}
\lim_{T\rightarrow\infty}\frac1T\sum_{t=0}^{T-1}\expect{y^n[t]}=\overline{f}^n,~~
\lim_{T\rightarrow\infty}\frac1T\sum_{t=0}^{T-1}\expect{\mathbf{z}^n[t]}=\overline{\mathbf{g}}^n,
\end{align*}
for any $n\in\{1,2,\cdots,N\}$. Since $f^*$ is the optimal solution to \eqref{prob-1}-\eqref{prob-2} over all algorithms, it follows for any $(\overline{f}^n,\overline{\mathbf{g}}^n)\in\mathcal P^n$, $n\in\{1,2,\cdots,N\}$ satisfying $\sum_{n=1}^N\overline{g}^n_l\leq d_l,~\forall l\in\{1,2,\cdots,L\}$, we have
$\sum_{n=1}^N\overline{f}^n\geq f^*$, and the second part of the lemma is proved.
\end{proof}

\bibliographystyle{imsart-number}
\bibliography{asyn-theory}

\begin{thebibliography}{21}

\bibitem{Al99}
\begin{bbook}[author]
\bauthor{\bsnm{Altman},~\bfnm{E.}\binits{E.}}
(\byear{1999}).
\btitle{Constrained Markov decision processes}.
\bpublisher{Chapman and Hall/CRC Press}.
\end{bbook}
\endbibitem

\bibitem{Be09}
\begin{bbook}[author]
\bauthor{\bsnm{Bertsekas},~\bfnm{D.}\binits{D.}}
(\byear{2009}).
\btitle{Convex Optimization Theory}.
\bpublisher{Athena Scientific}.
\end{bbook}
\endbibitem

\bibitem{Be01}
\begin{bbook}[author]
\bauthor{\bsnm{Bertsekas},~\bfnm{D.~P.}\binits{D.~P.}}
(\byear{2001}).
\btitle{Dynamic Programming and Optimal Control, 2nd edition, Vol. I}.
\bpublisher{Athena Scientific, Nashua, NH}.
\end{bbook}
\endbibitem

\bibitem{BT97}
\begin{bbook}[author]
\bauthor{\bsnm{Bertsekas},~\bfnm{D.~P.}\binits{D.~P.}} \AND
  \bauthor{\bsnm{Tsitsiklis},~\bfnm{J.~N.}\binits{J.~N.}}
(\byear{1997}).
\btitle{Parallel and Distributed Computation: Numerical Methods}.
\bpublisher{Athena Scientific, Nashua, NH}.
\end{bbook}
\endbibitem

\bibitem{BV04}
\begin{bbook}[author]
\bauthor{\bsnm{Boyd},~\bfnm{S.}\binits{S.}} \AND
  \bauthor{\bsnm{Vandenberghe},~\bfnm{L.}\binits{L.}}
(\byear{2004}).
\btitle{Convex Optimization}.
\bpublisher{Cambridge University Press}.
\end{bbook}
\endbibitem

\bibitem{BGPS06}
\begin{barticle}[author]
\bauthor{\bsnm{Byod},~\bfnm{S.}\binits{S.}},
  \bauthor{\bsnm{Ghosh},~\bfnm{A.}\binits{A.}},
  \bauthor{\bsnm{Prabhakar},~\bfnm{B.}\binits{B.}} \AND
  \bauthor{\bsnm{Shah},~\bfnm{D.}\binits{D.}}
(\byear{2006}).
\btitle{Randomized gossip algorithms}.
\bjournal{IEEE/ACM Transactions on Networking}
\bvolume{14,}
\bpages{2508-2530}.
\end{barticle}
\endbibitem

\bibitem{Durrett}
\begin{bbook}[author]
\bauthor{\bsnm{Durrett},~\bfnm{R.}\binits{R.}}
(\byear{2013}).
\btitle{Probability: Theory and Examples, 4th edition}.
\bpublisher{Cambridge University Press}.
\end{bbook}
\endbibitem

\bibitem{Fo66}
\begin{barticle}[author]
\bauthor{\bsnm{Fox},~\bfnm{B.}\binits{B.}}
(\byear{1966}).
\btitle{Markov renewal programming by linear fractional programming}.
\bjournal{SIAM Journal on Applied Mathematics}
\bvolume{14,}
\bpages{1418-1432}.
\end{barticle}
\endbibitem

\bibitem{GDHS13}
\begin{barticle}[author]
\bauthor{\bsnm{Gandhi},~\bfnm{A.}\binits{A.}},
  \bauthor{\bsnm{Doroudi},~\bfnm{S.}\binits{S.}},
  \bauthor{\bsnm{Harchol-Balter},~\bfnm{M.}\binits{M.}} \AND
  \bauthor{\bsnm{Scheller-Wolf},~\bfnm{A.}\binits{A.}}
(\byear{2013}).
\btitle{Exact analysis of the M/M/k/setup class of Markov chains via recursive
  renewal reward}.
\bjournal{Proc. ACM Sigmetrics}
\bpages{153-166}.
\end{barticle}
\endbibitem

\bibitem{LN14}
\begin{barticle}[author]
\bauthor{\bsnm{Li},~\bfnm{C.}\binits{C.}} \AND
  \bauthor{\bsnm{Neely},~\bfnm{M.~J.}\binits{M.~J.}}
(\byear{2014}).
\btitle{Solving convex optimization with side constraints in a multi-class
  queue by adaptive $c\mu$ rule}.
\bjournal{Queueing System}
\bvolume{77,}
\bpages{331-372}.
\end{barticle}
\endbibitem

\bibitem{Neely2010}
\begin{bbook}[author]
\bauthor{\bsnm{Neely},~\bfnm{M.~J.}\binits{M.~J.}}
(\byear{2010}).
\btitle{Stochastic Network Optimization with Application to Communication and
  Queueing Systems}.
\bpublisher{Morgan \& Claypool}.
\end{bbook}
\endbibitem

\bibitem{Neely12}
\begin{barticle}[author]
\bauthor{\bsnm{Neely},~\bfnm{M.~J.}\binits{M.~J.}}
(\byear{2012}).
\btitle{Asynchronous Scheduling for Energy Optimality in Systems with Multiple
  Servers}.
\bjournal{Proceedings of 46th Annual Conference on Information Sciences and
  Systems (CISS)}.
\end{barticle}
\endbibitem

\bibitem{Ne12}
\begin{barticle}[author]
\bauthor{\bsnm{Neely},~\bfnm{M.~J.}\binits{M.~J.}}
(\byear{2012}).
\btitle{Asynchronous Control for Coupled Markov Decision Systems}.
\bjournal{Information Theory Workshop (ITW)}.
\end{barticle}
\endbibitem

\bibitem{Ne09}
\begin{barticle}[author]
\bauthor{\bsnm{Neely},~\bfnm{M.~J.}\binits{M.~J.}}
(\byear{2013}).
\btitle{Dynamic Optimization and Learning for Renewal Systems}.
\bjournal{IEEE Transactions on Automatic Control}
\bvolume{58,}
\bpages{32-46}.
\end{barticle}
\endbibitem

\bibitem{PXYY15}
\begin{barticle}[author]
\bauthor{\bsnm{Peng},~\bfnm{Z.}\binits{Z.}},
  \bauthor{\bsnm{Xu},~\bfnm{Y.}\binits{Y.}},
  \bauthor{\bsnm{Yan},~\bfnm{M.}\binits{M.}} \AND
  \bauthor{\bsnm{Yin},~\bfnm{W.}\binits{W.}}
(\byear{2016}).
\btitle{ARock: an Algorithmic Framework for Asynchronous Parallel Coordinate
  Updates}.
\bjournal{To appear in SIAM Journal on Scientific Computing}.
\end{barticle}
\endbibitem

\bibitem{rockafellar2015convex}
\begin{bbook}[author]
\bauthor{\bsnm{Rockafellar},~\bfnm{Ralph~Tyrell}\binits{R.~T.}}
(\byear{2015}).
\btitle{Convex analysis}.
\bpublisher{Princeton university press}.
\end{bbook}
\endbibitem

\bibitem{Ro02}
\begin{bbook}[author]
\bauthor{\bsnm{Ross},~\bfnm{S.}\binits{S.}}
(\byear{2002}).
\btitle{Introduction to Probability Models, 8th edition}.
\bpublisher{Academic Press}.
\end{bbook}
\endbibitem

\bibitem{Sc83}
\begin{barticle}[author]
\bauthor{\bsnm{Schaible},~\bfnm{S.}\binits{S.}}
(\byear{1983}).
\btitle{Fractional programming}.
\bjournal{Zeitschrift fur Operations Research}
\bvolume{27,}
\bpages{39-54}.
\end{barticle}
\endbibitem

\bibitem{SN11}
\begin{barticle}[author]
\bauthor{\bsnm{Srivastava},~\bfnm{K.}\binits{K.}} \AND
  \bauthor{\bsnm{Nedic},~\bfnm{A.}\binits{A.}}
(\byear{2011}).
\btitle{Distributed Asynchronous Constrained Stochastic Optimization}.
\bjournal{IEEE Journal of Selected Topics in Signal Processing}
\bvolume{5,}
\bpages{772-790}.
\end{barticle}
\endbibitem

\bibitem{WN15}
\begin{barticle}[author]
\bauthor{\bsnm{Wei},~\bfnm{X.}\binits{X.}} \AND
  \bauthor{\bsnm{Neely},~\bfnm{M.~J.}\binits{M.~J.}}
(\byear{2017}).
\btitle{Data Center Server Provision: Distributed Asynchronous Control for
  Coupled Renewal Systems}.
\bjournal{IEEE/ACM Transactions on Networking}
\bvolume{to appear}.
\end{barticle}
\endbibitem

\bibitem{Yao02}
\begin{barticle}[author]
\bauthor{\bsnm{Yao},~\bfnm{D.~D.}\binits{D.~D.}}
(\byear{2002}).
\btitle{Dynamic Scheduling via Polymatroid Optimization}.
\bjournal{Proceeding Performance Evaluation of Complex Systems: Techniques and
  Tools}
\bpages{89-113}.
\end{barticle}
\endbibitem

\end{thebibliography}

\end{document}